\documentclass[11pt]{amsart}
\usepackage{a4wide}
\usepackage{comment}
\usepackage[dvipsnames]{xcolor}
\usepackage{mathrsfs}
\usepackage{amsmath}
\usepackage{amsthm}
\usepackage{amsfonts}
\usepackage{amssymb}
\usepackage{enumitem}
\usepackage{graphicx}
\usepackage{url}
\usepackage{tikz}
\usepackage[T1]{fontenc}
\usepackage{newpxmath}
\usepackage{overpic}
\usepackage{hyperref}
\usepackage[noabbrev,capitalize,nameinlink]{cleveref}

\newtheorem{theorem}{Theorem}[section]
\newtheorem{proposition}[theorem]{Proposition}
\newtheorem{lemma}[theorem]{Lemma}

\theoremstyle{definition}
\newtheorem{definition}[theorem]{Definition}
\newtheorem{example}[theorem]{Example}

\theoremstyle{remark}
\newtheorem{remark}[theorem]{Remark}
\numberwithin{equation}{section}

\newcommand{\norm}[1]{\left\lVert#1\right\rVert}
\usepackage{mleftright}

\newcommand{\R}{\mathbb{R}}

\newcommand{\D}{\mathbb{D}}

\newcommand{\ve}{\varepsilon}

\definecolor{darkgreen}{RGB}{0,153,0}
\definecolor{darkred}{RGB}{204,0,0}
\definecolor{darkblue}{RGB}{0,51,204}
\definecolor{red}{RGB}{242,43,29}
\hypersetup{colorlinks,linkcolor={blue},citecolor={darkgreen},urlcolor={darkgreen}}

\begin{document}

\title{Folded symplectic forms in contact topology}
\author{Joseph Breen}
\address{University of Iowa, Iowa City, IA 52240}
\email{joseph-breen-1@uiowa.edu.edu} \urladdr{https://sites.google.com/view/joseph-breen/home}
\thanks{The author was partially supported by NSF Grant DMS-2038103.}

\begin{abstract}
    We establish the relationship between folded symplectic forms and convex hypersurface theory in contact topology. As an application, we use convex hypersurface theory to reprove and strengthen the existence result for folded symplectic forms due to Cannas da Silva, and we generalize to all even dimensions Baykur's $4$-dimensional existence result of folded Weinstein structures and folded Lefschetz fibrations. 
\end{abstract}

\maketitle

\tableofcontents

\section{Introduction}

Let $\Sigma$ be a closed and oriented manifold of dimension $2n\geq 2$. We are interested in two structures associated to such a manifold, the first being that of a \textit{folded symplectic form}, and the second being the structure induced by viewing $\Sigma$ as a \textit{convex hypersurface} embedded in a contact manifold. 

\begin{definition}\label{def:folded_form}
A \textbf{folded symplectic form} is a closed $2$-form $\omega\in \Omega^2(\Sigma)$ such that
\begin{enumerate}
    \item $\omega^n \in \bigwedge^{2n} T^*\Sigma$ intersects the $0$-section transversally along the \textbf{fold} $\Gamma := \{\omega^n = 0\}$, and 
    \item $\iota^*\omega$ has maximal rank, where $\iota:\Gamma\hookrightarrow \Sigma$ is the inclusion. 
\end{enumerate}
The tuple $(\Sigma, \omega)$ is a \textbf{folded symplectic manifold}. We call $R_{\pm}:= \{\pm \omega^n >0\}$ the \textbf{positive region} and \textbf{negative region}, respectively.  
\end{definition}

Informally, a folded symplectic form is a $2$-form that degenerates as mildly as possible along a codimension-$1$ hypersurface, and is otherwise (anti)symplectic. Indeed, $(R_{+}, \omega\mid_{R_{+}})$ and $(-R_{-}, \omega\mid_{R_{-}})$ are open symplectic manifolds, where $-R_-$ denotes $R_-$ equipped with the opposite orientation that it inherits from $\Sigma$. 

\begin{definition}\label{def:convex}
A co-orientable hypersurface  $\iota: \Sigma \hookrightarrow (M^{2n+1}, \xi = \ker \alpha)$ in a contact manifold is \textbf{convex} if there is a contact vector field $X$, i.e. a vector field satisfying $\mathcal{L}_X\xi = \xi$, everywhere transverse to $\Sigma$. We call $R_{\pm} := \{\pm\alpha(X) > 0\}\subset \Sigma$ the \textbf{positive region} and \textbf{negative region}, respectively, and we call the codimension-$1$ contact-type submanifold $\Gamma := \{\alpha(X) = 0\}\subset \Sigma$ the \textbf{dividing set}.    
\end{definition}

Convex hypersurfaces form a generic class of hypersurfaces in contact manifolds. Since $(M, \xi=\ker \alpha)$ is co-oriented, the co-orientation given by $X$ induces an orientation on $\Sigma$. Importantly, $R_+$ is an exact symplectic manifold with a primitive given by an appropriate positive rescaling of $\iota^*\alpha\mid_{R_+}$, and $(-1)^{n+1}R_-$ is likewise exact symplectic via an appropriate negative rescaling of $\iota^*\alpha\mid_{R_-}$. Both $\overline{R}_+$ and $(-1)^{n+1}\overline{R}_-$ are symplectic fillings of the contact dividing set $\Gamma$. 

\subsection{Context}\label{subsec:context}

The goal of this note is to establish the connection between folded symplectic forms and convex hypersurfaces. Before stating our main results, \cref{theorem:main0}, \cref{theorem:main}, and \cref{thm:lefschetz}, we include a brief discussion of each theory for context. 

\subsubsection{Folded symplectic forms}\label{subsec:context_folded}

Folded symplectic forms were first introduced by Martinet in his thesis \cite{Martinet1970}, where he proved a Darboux theorem; namely, that around any point $p\in \Gamma$ there are local coordinates such that $p=0$, the fold is given by $\Gamma = \{y_1=0\}$, and\footnote{Here we are locally orienting $\Sigma$ via the top wedge power of $\sum_{j=1}^n dx_j \wedge dy_j$, so that $R_+ = \{y_1 > 0\}$.}
\begin{equation}
\omega = y_1\, dx_1\wedge dy_1 + \sum_{j=2}^n dx_j \wedge dy_j.   
\end{equation}
Later, folded symplectic forms were thoroughly investigated in a series of papers by Cannas da Silva, Guillemin, Pires, and Woodward \cite{cannas2000unfolding,cannas2010foldedfour,cannas2011origami}. Cannas da Silva \cite{cannas2010foldedfour} established existence in all even dimensions, in particular showing that folded symplectic forms satisfy an existence $h$-principle. The sufficient formal data on $\Sigma$ necessary for existence is that of a stable almost complex structure, i.e., a complex vector bundle structure on $T\Sigma \oplus \underline{\ve}^2$, where $\underline{\ve}^2 \to \Sigma$ is the trivial rank-$2$ vector bundle. Stable almost complex structures, and hence folded symplectic forms, exist on all closed and oriented $4$-manifolds. 

In \cite{baykur2006kahler}, Baykur upgraded the existence of folded symplectic forms on $4$-manifolds to that of, what we will call in this paper, \textit{folded Weinstein structures}.\footnote{Baykur worked with Stein structures. We will work directly with the symplectic counterpart; see \cite{cieliebak2012stein}.} Briefly, a folded Weinstein manifold is an \textit{exact} folded symplectic manifold where the positive and negative regions are each Weinstein manifolds inducing the same contact structure on the fold. A precise definition is given in \cref{def:folded_weinstein}; see also \cref{remark:history} for other uses of \textit{folded Weinstein} language in the literature. 

We summarize the status (prior to this paper) of existence of folded symplectic structures as follows. 

\begin{theorem}\label{theorem:old}
Let $\Sigma$ be a closed and oriented manifold of dimension $2n\geq 2$. 
\begin{enumerate}
    \item \cite{cannas2010foldedfour} If $\Sigma$ admits a stable almost complex structure, then $\Sigma$ admits a folded symplectic structure. In particular, if $2n=4$, then $\Sigma$ admits a folded symplectic structure.
    \item \cite{baykur2006kahler} If $2n=4$, then $\Sigma$ admits a folded Weinstein structure. 
\end{enumerate}
\end{theorem}

Folded symplectic manifolds have been studied with toric actions \cite{hockensmith2015classification,lee2015folded}, have been studied dynamically \cite{cardona2021integrable}, and have even been probed by pseudoholomorphic curves in dimension $4$ \cite{vonBergmann2007pseudoholomorphic}. They have also been studied in terms of their relationship with other singular symplectic manifolds, such as $b$-symplectic manifolds \cite{Guillemin_2014,guillemin2017desingularizing}, and, recently, odd-dimensional counterparts (folded contact forms and $b$-contact forms) have been introduced in the literature \cite{miranda2023contact, cardona2023existence}.

\subsubsection{Convex hypersurface theory}\label{subsec:context_convex}

Convex hypersurface theory was introduced by Giroux \cite{giroux1991convexite}, who established a $C^{\infty}$-genericity statement in dimension $3$, among other foundational results. In the following decades, convex surface theory was used to great effect in $3$-dimensional contact topology, for instance in the classification of tight contact structures \cite{honda2000classification,honda2000classification2}. In higher dimensions the theory poses more technical difficulties, and it was not until recently that a $C^0$-genericity result was first established by Honda and Huang. In fact, they proved that convex hypersurfaces with \textit{Weinstein} positive and negative regions are $C^0$-generic. 

\begin{theorem}[\cite{honda2019convex}, see also \cite{eliashberg2022hondahuangs}]\label{theorem:hh19}
Let $\Sigma \hookrightarrow (M^{2n+1}, \xi)$ be a closed and co-oriented hypersurface in a co-oriented contact manifold. Then there is a $C^0$-small perturbation of $\Sigma$ so that it is convex and such that $R_+$ and $(-1)^{n+1}R_-$ are exact symplectic manifolds that naturally inherit Weinstein structures.     
\end{theorem}

Honda and Huang's further work \cite{HH18} and work together with the author \cite{BHH23}, alongside work of others like Eliashberg and Pancholi \cite{eliashberg2022hondahuangs} and Sackel \cite{sackel2019handle}, has since laid the groundwork for the use of convex hypersurface theory in the comparatively unexplored frontier of higher-dimensional contact topology.

\subsection{Main results}

\cref{def:folded_form} and \cref{def:convex} suggest a connection between folded symplectic forms and convex hypersurface theory. However, to the best of the author's knowledge, the direct relationship has not been suggested. For instance, there is no mention of convex hypersurface theory in any of the folded symplectic literature cited in \cref{subsec:context_folded}, nor are folded symplectic forms named in any of the classical or modern literature on convex hypersurface theory cited in \cref{subsec:context_convex}. Convex hypersurface theory is used in the recent papers on singular contact forms \cite{miranda2023contact,cardona2023existence}, but only for the purpose of the singular contact structures; there is no indication that convexity induces a folded symplectic structure on the hypersurface.\footnote{For more commentary on relevant literature, see \cref{remark:history} below for a discussion on Honda and Huang's use of \textit{folded} terminology in \cite{honda2019convex}, and \cref{remark:salamon} for reference to an intrinsic form of convex hypersurface theory described by Salamon based on a series of lectures given by Eliashberg \cite{salamon2022notes}.}

Our first result establishes this relationship explicitly. The precise definition of \textit{positive contact-type Liouville form} used in the statement of the theorem is \cref{def:CTliouville}. 

\begin{theorem}\label{theorem:main0}
Let $\Sigma$ be a closed and oriented manifold of dimension $2n\geq 2$, and $\Gamma\subset \Sigma$ a separating hypersurface. The following statements are equivalent: 
\begin{enumerate}
    \item The manifold $\Sigma$ admits an exact folded symplectic form $\omega$ with positive contact-type Liouville form $\lambda$ and contact fold $(\Gamma, \xi_{\Gamma})$.
    \item There is a contact structure on $\Sigma \times \R$ such that $\Sigma \times \{0\}$ is convex with contact dividing set $(\Gamma, \xi_{\Gamma})$. 
\end{enumerate}
\end{theorem}

As an application, we leverage convex hypersurface theory to upgrade the existence statements of folded symplectic structures to that of folded Weinstein structures in all dimensions. Our second main result is the following strengthening of \cref{theorem:old} and is a consequence of \cref{theorem:main0} together with \cref{theorem:hh19}.

\begin{theorem}\label{theorem:main}
Let $\Sigma$ be a closed and oriented manifold of dimension $2n\geq 2$. If $\Sigma$ admits a stable almost complex structure, then $\Sigma$ admits a folded Weinstein structure.    
\end{theorem}

\begin{remark}
Our proof of \cref{theorem:main} recovers both (1) and (2) of \cref{theorem:old} with a different technique. The difference, particularly with Cannas da Silva's method in \cite{cannas2010foldedfour}, will be further elaborated in \cref{subsec:main}.
\end{remark}

In \cite{baykur2006kahler}, Baykur also considered other folded structures on $4$-manifolds such as \textit{folded Kähler structures} and \textit{folded Lefschetz fibrations} that naturally come with the folded Weinstein territory. We concern ourselves with the latter. Roughly, a folded Lefschetz fibration is obtained by gluing two Lefschetz fibrations with (symplectically) isotopic monodromy along their boundaries to obtain a closed manifold; a precise definition is \cref{def:AFWLF}. Our final main result is a generalization of Baykur's folded Lefschetz fibration existence result to all dimensions.

\begin{theorem}\label{thm:lefschetz}
Let $2n\geq 4$. Every folded Weinstein manifold $(\Sigma^{2n}, \lambda, \phi)$ admits a compatible folded Weinstein Lefschetz fibration. In particular, every closed and oriented manifold of dimension $2n$ with a stable almost complex structure admits a folded Weinstein Lefschetz fibration. 
\end{theorem}

\begin{remark}
Baykur remarks (see his discussion after \cite[Definition 6.5]{baykur2006kahler}) that, assuming future advances in Weinstein Lefschetz fibration technology in higher dimensions, many of his definitions and results generalize naturally. This advancement is now possible with work of Giroux and Pardon \cite{giroux2017lefschetz} and the work of the author with Honda and Huang \cite{BHH23}.     
\end{remark}

We close the introduction with two more literature remarks.

\begin{remark}[Folded Weinstein language due to Honda and Huang]\label{remark:history}
Honda and Huang \cite{honda2019convex} introduce a structure they call \textit{folded Weinstein} for hypersurfaces in contact manifolds, though this notion is a priori slightly different to how we use \textit{folded Weinstein} in this paper (i.e. \cref{def:folded_weinstein}). Namely, the folded Weinstein hypersurfaces of \cite{honda2019convex} are specifically hypersurfaces of contact manifolds, rather than manifolds with an intrinsic structure as in \cref{def:folded_form} or \cref{def:folded_weinstein}. Moreover, a folded Weinstein hypersurface of \cite{honda2019convex} is allowed to have many different folds that alternate between positive and negative cobordism regions, rather than one fold between two domains. However, there is no contradiction or confusion --- every folded Weinstein hypersurface in a contact manifold can be perturbed to a Weinstein convex hypersurface, and by our proof of \cref{theorem:main}, this induces on $\Sigma$ a folded Weinstein structure in the sense of \cref{def:folded_weinstein}. For this reason, our language is appropriate a posteriori.
\end{remark}

\begin{remark}\label{remark:salamon}
The notes of Salamon \cite{salamon2022notes}, based on a series of lectures given by Eliashberg in 2022 on the work of Eliashberg and Pancholi \cite{eliashberg2022hondahuangs}, discuss an intrinsic notion of convex hypersurface theory via pairs of differential forms on even-dimensional manifolds. It is likely that our work can be formulated in this framework.    
\end{remark}

\subsection{Organization} We review the basics of Weinstein topology, folded symplectic forms, convex hypersurface theory, and Lefschetz fibrations in \cref{sec:prelim}. In \cref{sec:exact}, we record the various definitions needed to make each of our main results precise, and provide some constructions and examples. Finally, in \cref{sec:proofs} we prove \cref{theorem:main0}, \cref{theorem:main}, and \cref{thm:lefschetz}.

\subsection{Acknowledgments} We thank Gabe Islambouli for an enlightening conversation that sparked this project, and in particular for introducing the author to the literature of folded symplectic forms. We also thank Ko Honda and Austin Christian for helpful comments. 
\section{Preliminaries}\label{sec:prelim}

Here we review the necessary background material. None of the content in this section is original.

\begin{remark}[A word on notation.]
Let $M$ be a manifold, let $\iota:M \hookrightarrow M \times \R$ be the inclusion into the $0$-level, let $\pi: M \times \R \to M$ be the projection, and let $\eta$ be a differential form on $M \times \R$. We are often interested in the form $\pi^*\iota^*\eta$ on $M\times \R$. To clean up the notation, we will write $\eta_{M} := \iota^*\eta$ and then will omit the $\pi^*$ when viewing the form on $M \times \R$. There should be no risk of confusion. 
\end{remark}

\subsection{Weinstein topology}

Here we review the basic definitions of Weinstein structures, mostly to establish notation conventions used in this paper. We also include a review of \textit{ideal} Liouville structures as described by Giroux \cite{giroux2020ideal}. There is much left unsaid about Weinstein manifolds; for more details, the standard reference is \cite{cieliebak2012stein}, and a nice survey article is \cite{eliashberg2017weinstein}.

\begin{definition}\label{def:liouville}
Let $V$ be a compact manifold of dimension $2n$ with nonempty boundary. A \textbf{Liouville (domain) structure} on $V$ is a $1$-form $\lambda\in \Omega^1(V)$ such that $d\lambda$ is symplectic and such that $\iota^*\lambda$ is a positive contact form, where $\iota:\partial V \hookrightarrow V$ is the inclusion and $\partial V$ is oriented as the boundary of $V$, oriented by its symplectic structure. The unique vector field $X_{\lambda}$ satisfying $X_{\lambda} \lrcorner d\lambda = \lambda$ is called the \textbf{Liouville vector field}. Equivalently, a Liouville (domain) structure is given by a pair $(\omega, X)$ where $\omega$ is a symplectic form, $X$ is a vector field satisfying $\mathcal{L}_X\omega = \omega$, and $X$ is outwardly transverse to $\partial X$. The data $(V, \lambda)$ or $(V, \omega, X)$ is called a \textbf{Liouville domain}. More generally, we may speak of a compact \textbf{Liouville cobordism} wherein the Liouville vector field $X_{\lambda}$ is allowed to point inward along a negative boundary component $\partial_- V$.
\end{definition}

Given $(\omega, X)$ the equivalence in \cref{def:liouville} comes from setting $\lambda := X \lrcorner\omega$. Indeed, if $X_{\lambda}$ is the Liouville vector field for $\lambda$, then $\mathcal{L}_{X_{\lambda}} d\lambda = d\lambda$, and the positive contact condition is equivalent to outward transversality at the boundary. 

\begin{definition}
Let $W$ be a compact manifold of dimension $2n$ with nonempty boundary. A \textbf{Weinstein structure} is a tuple $(\lambda, \phi)$ where
\begin{enumerate}
    \item $(W,\lambda)$ is a Liouville domain, 
    \item $\phi: W \to \R$ is a Morse function such that $\partial W = \phi^{-1}(c)$ is a regular level set, and 
    \item the Liouville vector field $X_{\lambda}$ is gradient-like for $\phi$. 
\end{enumerate}
Equivalently, a Weinstein structure on $W$ may be given by a triple $(\omega, X, \phi)$, where $(\omega, X)$ is a Liouville structure. The data $(W, \lambda, \phi)$ or $(W, \omega, X, \phi)$ is a \textbf{Weinstein domain}. A \textbf{Weinstein homotopy} is simply a $1$-parameter family $(W, \lambda_t, \phi_t)$ of Weinstein structures on a fixed domain $W$. We often abuse language and say that two Weinstein structures on different but diffeomorphic manifolds are \textbf{Weinstein homotopic} if they are Weinstein homotopic under the pullback via a diffeomorphism.  
\end{definition}

The Liouville and Weinstein structures discussed so far by definition reside on compact domains. It is also natural to consider the structures on open manifolds with cylindrical ends. For our purposes, we simply record the following definition. 

\begin{definition}
Let $(V, \lambda)$ be a Liouville domain. The \textbf{completion} of $(V,\lambda)$ is the data $(\hat{V}, \hat{\lambda})$ where $\hat{V} := V \cup (\partial V \times [0,\infty)_s)$ and $\hat{\lambda}$ is the extension of $\lambda$ to $\hat{V}$ via $e^s\, \alpha$, where $\alpha$ is the induced contact form on $\partial V$. We call $(\hat{V}, \hat{\lambda})$ a (\textbf{completed, finite-type}) \textbf{Liouville manifold}.     
\end{definition}

There is one last type of Liouville structure to discuss which has the benefit of living on a compact domain, but simultaneously behaving like a completion. 

\begin{definition}\label{def:ideal_domain}
Let $V$ be a compact manifold of dimension $2n$ with nonempty boundary, and let $\mathring{V} = V \setminus \partial V$ denote the interior of $V$. An \textbf{ideal Liouville structure} on $V$ is an exact symplectic form $\omega$ on $\mathring{V}$ admitting a primitive $\lambda$ such that there is a smooth function $u: V \to [0,\infty)$ with $\partial V = u^{-1}(0)$ a regular level set for which the $1$-form $u\, \lambda$ on $\mathring{V}$ extends to a $1$-form on all of $V$, inducing a contact form on $\partial V$. Such a $1$-form $\lambda$ is called an \textbf{ideal Liouville form} and the pair $(V, \omega)$ is an \textbf{ideal Liouville domain}.      
\end{definition}

Ideal Liouville domains exhibit a number of nice properties over normal Liouville domains or even completed finite-type Liouville manifolds. For example, the ideal Liouville structure $\omega$ --- independent of the choice of primitive --- induces a positive contact structure on $\partial V$ \cite[Proposition 2]{giroux2020ideal}. 

\begin{example}[Ideal completion]
Let $(V, \lambda_0)$ be a compact Liouville domain of dimension $2n$. Via the flow of the the Liouville vector field we may identify a collar neighborhood of the boundary $N(\partial V) \cong [-\ve, 0]_s \times \partial V$ with $\lambda_0 = e^s\, \alpha_0$, where $\alpha_0$ is the contact form induced on $\partial V$ by $\lambda_0$. Let $u:N(\partial V) \to [0,1]$ be a smooth function $u=u(s)$ such that $u(0) = 0$, $u'(s) < 0$ for $-\ve < s\leq 0$, $u(-\ve) = 1$, and $u'(-\ve) = 0$. Extend $u$ to a smooth function on $V_0$ by the constant function $u\equiv 1$. Finally, define a $1$-form on $\mathring{V}$ by $\lambda:= \frac{1}{u}\, \lambda_0$. 

We claim that $\omega := d(\frac{1}{u}\, \lambda_0)$ is an ideal Liouville structure on $V$. Indeed, on $\mathring{N}(\partial V):=[-\ve,0)_s\times \partial V$ we have $\omega =e^s\frac{u(s) - u'(s)}{u(s)^2}\, ds\wedge \alpha_0 + \frac{1}{u(s)}\, d\alpha_0$ and so 
\[
\omega^n = e^{ns}\frac{u(s) - u'(s)}{u(s)^{n+1}}\, ds\wedge \alpha_0\wedge (d\alpha_0)^{n-1} > 0.
\]
Moreover, the Liouville vector field in $\mathring{N}(\partial V)$ is 
\[
X_{\lambda} = e^{-s}\frac{u(s)^2}{u(s)-u'(s)}\, \partial_s
\]
which is positively parallel to $\partial_s$ and is complete (in forward time) on $\mathring{N}(\partial V)$, hence $(\mathring{V}, \lambda)$ is symplectomorphic to the completion of the Liouville domain $(V, \lambda_0)$. We call the ideal Liouville domain $(V, \omega)$ the \textbf{ideal completion} of $(V, \lambda_0)$. 
\end{example}

\subsection{Folded symplectic forms}

We only need one nontrivial fact about folded symplectic forms, which is a local normal form near the fold (\cref{lemma:fold_normal}). For more details on folded symplectic forms, we refer the reader to \cite{cannas2000unfolding,cannas2010foldedfour}. We begin by restating the main definition from the introduction for convenience. 

\begin{definition}
Let $\Sigma$ be a closed and oriented manifold of dimension $2n$. A \textbf{folded symplectic form} is a closed $2$-form $\omega\in \Omega^2(\Sigma)$ such that
\begin{enumerate}
    \item $\omega^n \in \bigwedge^{2n} T^*\Sigma$ intersects the $0$-section transversally along the \textbf{fold} $\Gamma:= \{\omega^n = 0\}$, and 
    \item $\iota^*\omega$ has maximal rank, where $\iota:\Gamma\hookrightarrow \Sigma$ is the inclusion. 
\end{enumerate}
The tuple $(\Sigma, \omega)$ is a \textbf{folded symplectic manifold}. We call $R_{\pm}:= \{\pm \omega^n >0\}$ the \textbf{positive region} and \textbf{negative region}, respectively.  
\end{definition}

Since $\omega_{\Gamma}:=\iota^*\omega$ has maximal rank, $\ker \omega_{\Gamma}$ is a $1$-dimensional subbundle of $T\Gamma \to \Gamma$ called the \textit{null-foliation} of the folded symplectic form. Moreover, $\omega_{\Gamma}^{n-1}$ naturally orients the quotient bundle $T\Gamma / \ker \omega_{\Gamma}$. Since $\Gamma$ is oriented, there is an induced orientation on the line bundle $\ker\omega_{\Gamma} \to \Gamma$. In particular, it admits nonvanishing sections. With this in mind, the following lemma gives a local normal form for a folded symplectic form near the fold. 

\begin{lemma}[\cite{cannas2000unfolding}]\label{lemma:fold_normal}
Let $(\Sigma, \omega)$ be a folded symplectic manifold with fold $\Gamma$. Let $\iota:\Gamma \hookrightarrow \Sigma$ be the inclusion, let $\omega_{\Gamma} := \iota^*\omega$, let $v$ be a non-vanishing section that orients the line bundle $\ker \omega_{\Gamma} \to \Gamma$, and let $\alpha\in \Omega^1(\Gamma)$ be a $1$-form with $\alpha(v) = 1$. Then near $\Gamma$ there are local coordinates on a neighborhood $N(\Gamma) \cong \Gamma \times (-\ve, \ve)_{\tau}$ of $\Gamma$ in $\Sigma$ such that, on $N(\Gamma)$, we have 
\begin{equation}\label{eq:fold_normal}
    \omega = \omega_{\Gamma} - d(\tau^2\, \alpha).
\end{equation}
\end{lemma}

\begin{remark}[Orientation conventions]\label{remark:orientation}
The statement of \cref{lemma:fold_normal} differs from that of \cite[Theorem 1]{cannas2000unfolding} by a negative sign because of a difference in orientation convention. In our paper, $N(\Gamma)\cong \Gamma \times (-\ve, \ve)_{\tau}$ is oriented via $\Omega_{\Gamma} \wedge d\tau$, where $\Omega_{\Gamma}$ is an orienting volume form on $\Gamma$, so that $R_+ \cap N(\Gamma) = \{\tau > 0\}$. Indeed, from \eqref{eq:fold_normal} we have $\omega = \omega_{\Gamma} - 2\tau\, d\tau + \tau^2\, d\alpha$ and so 
\begin{equation}
    \omega^n = 2n\tau\left(\alpha \wedge \omega_{\Gamma}^{n-1}\wedge d\tau \, + \, \tau^2\, \alpha \wedge (d\alpha)^{n-1}\wedge d\tau\right).
\end{equation}
By construction, $\alpha \wedge \omega_{\Gamma}^{n-1}$ orients $\Gamma$ and so $\alpha \wedge \omega_{\Gamma}^{n-1}\wedge d\tau$ orients $N(\Gamma)$. It follows that $\omega^n > 0$ for $\tau > 0$.    
\end{remark}

\subsection{Convex hypersurface theory}

References for a more thorough background on convex hypersurface theory are \cite{giroux1991convexite,HH18,sackel2019handle}; see also Salamon's notes \cite{salamon2022notes} on the work of Eliashberg and Pancholi \cite{eliashberg2022hondahuangs}. Again we repeat the main definition for convenience. 

\begin{definition}
Let $(M, \xi=\ker \alpha)$ be a co-oriented contact manifold. A smoothly embedded closed and co-orientable hypersurface $\Sigma \subset M$ is \textbf{convex} if there is a contact vector field $X$ everywhere transverse to $\Sigma$. We define the \textbf{dividing set} to be $\Gamma^X:= \{\alpha(X) = 0\}$, and the \textbf{positive (resp. negative) region} to be $R_{\pm}^X := \{\pm\alpha(X) > 0\}$.    
\end{definition}

\begin{remark}
Note that $\Gamma^X$ and $R_{\pm}^X$ depend on the choice of contact vector field. However, the (contact) isotopy class of $\Gamma^X$ in $\Sigma$ is independent of $X$ and so we will typically omit the notational dependence on the choice of contact vector field.      
\end{remark}

Integrating the transverse contact vector field $X$ (and decaying it with a contact Hamiltonian) gives an arbitrarily large vertically-invariant standard neighborhood of a convex hypersurface. 

\begin{lemma}[\cite{giroux1991convexite}]\label{lemma:contactnbd}
Let $\Sigma \subset (M, \xi)$ be a convex hypersurface. Then there is a neighborhood $N(\Sigma) \cong \Sigma \times \R_t$ of $\Sigma$ in $M$ where $\Sigma = \Sigma\times \{0\}$ and $\xi = \ker(f\, dt + \beta)$ for a smooth $t$-independent function $f: \Sigma \to \R$ and a $t$-independent $1$-form $\beta \in \Omega^1(\Sigma)$. Moreover, the contact condition is 
\begin{equation}
    dt\wedge \left(f\, (d\beta)^n - n\, df\wedge \beta \wedge (d\beta)^{n-1}\right) > 0.
\end{equation}
In particular, $f\, (d\beta)^n - n\, df\wedge \beta \wedge (d\beta)^{n-1}$ is an orienting volume form on $\Sigma$. 
\end{lemma}

The following lemma allows us to further localize near the dividing set of a convex hypersurface. 

\begin{lemma}[\cite{giroux1991convexite}]\label{lemma:dividingsetnormal}
Let $\Sigma \subset (M, \xi)$ be a closed and oriented convex hypersurface with dividing set $\Gamma$. Let $\iota:\Gamma \to M$ be the inclusion. After a contact isotopy rel $\Sigma$, there are coordinates on a neighborhood $N(\Gamma) \cong \Gamma \times (-\ve, \ve)_{\tau}$ and a contact form $\alpha$ on $N(\Gamma) \times \R_t$ such that 
\begin{equation}
    \alpha = \tau\, dt + \alpha_{\Gamma}
\end{equation}
where $\alpha_{\Gamma}:= \iota^*\alpha$ is the induced contact form on $(\Gamma, \xi_{\Gamma})$. 
\end{lemma}

In the introduction it was stated that the positive and negative regions of a convex hypersurface are naturally exact symplectic fillings of the contact-type dividing set. The precise statement is given in the language of ideal Liouville domains.

\begin{proposition}[\cite{giroux1991convexite}]\label{prop:CHT_ideal}
Let $(\Sigma \times \R_t, \ker(f\, dt + \beta))$ be a vertically-invariant neighborhood of a convex hypersurface $\Sigma = \Sigma \times \{0\}$. Let $\lambda_{\pm}:= \frac{1}{f}\, \beta\mid_{R_{\pm}}$. Then 
\[
\left((\pm 1)^{n+1}\overline{R}_{\pm},\, d\lambda_{\pm}\right)
\]
is an ideal Liouville domain. Moreover, the characteristic foliation on $R_{+}$ is directed by the (ideal) Liouville vector field $X_{\lambda_+}$, and the characteristic foliation on $R_-$ is directed by the negative (ideal) Liouville vector field $-X_{\lambda_-}$.
\end{proposition}

\begin{definition}
Let $\Sigma \subset (M,\xi=\ker \alpha)$ be a co-oriented (not necessarily convex) hypersurface in a co-oriented contact manifold. The \textbf{characteristic foliation} is the oriented singular foliation given by $\Sigma_{\xi}:= (T\Sigma \cap \xi)^{\bot}$, where $\bot$ is the symplectic orthogonal complement given by the conformal symplectic structure on $\xi$.     
\end{definition}

The orientation of the foliation is induced by the orientation of $\Sigma$ and the co-orientation of $\xi$. Singular points $p$ of the characteristic foliation $\Sigma_{\xi}$, i.e., points where $T_p\Sigma = \xi_p$, are, by definition, positive (resp. negative) according to when the co-orientation of the contact plane $\xi_p$ agrees (resp. disagrees) with the co-orientation of $T_p\Sigma$. 

\begin{definition}[\cite{honda2019convex}]
Let $(M, \xi=\ker \alpha)$ be a co-oriented contact manifold. A \textbf{Weinstein convex hypersurface} is a tuple $(\Sigma, \phi)$ where $\Sigma \subset M$ is a convex hypersurface and $\phi:\Sigma \to \R$ is a Morse function for which the characteristic foliation $\Sigma_{\xi}$ is gradient-like (by which we mean there is some vector field $Y$ directing $\Sigma_{\xi}$ which is gradient-like for $\phi$).   
\end{definition}

The following fact then follows from \cref{prop:CHT_ideal}, which justifies the language of the previous definition. 

\begin{proposition}
Let $(\Sigma \times \R_t, \ker(f\, dt + \beta))$ be a vertically-invariant neighborhood of a Weinstein convex hypersurface $(\Sigma, \phi)$. Let $\lambda_{\pm}:= \frac{1}{f}\, \beta\mid_{R_{\pm}}$ and let $\phi_{\pm}:=\pm \phi\mid_{R_{\pm}}$. Then $((\pm 1)^{n+1}R_{\pm}, \lambda_{\pm}, \phi_{\pm})$ is a completed finite-type Weinstein manifold and $((\pm 1)^{n+1}\overline{R}_{\pm}, d\lambda_{\pm})$ is its ideal completion. 
\end{proposition}

\begin{example}[The standard convex sphere]\label{ex:convex_sphere}
Let $(M, \xi) = (\R^{2n+1}, \ker \alpha)$, where 
\[
\alpha = dz + \frac{1}{2}\sum_{j=1}^{2n} (x_j\, dy_j - y_j\, dx_j)
\]
is the standard radial contact form, and let $S^{2n}=\{\norm{\mathbf{x}}^2 + \norm{\mathbf{y}}^2 + z^2 = 1\}$ be the unit sphere with outward co-orientation. Let 
\[
X = z\, \partial_z + \frac{1}{2}\sum_{j=1}^n (x_j\, \partial_{x_j} + y_j\, \partial_{y_j})
\]
be the radial vector field. Note that $\mathcal{L}_{z\, \partial_z}(dz) = dz$ and 
\[
\mathcal{L}_{\frac{1}{2}(x_j\, \partial_{x_j} + y_j\, \partial_{y_j})}\frac{1}{2}(x_j\, dy_j - y_j\, dx_j) = \frac{1}{2}(x_j\, \partial_{x_j} + y_j\, \partial_{y_j})  \, \lrcorner\,  (dx_j\wedge dy_j) = \frac{1}{2}(x_j\, dy_j - y_j\, dx_j)
\]
and so $\mathcal{L}_X\alpha = \alpha$. Thus, $X$ is a (strict) contact vector field and $S^{2n}$ is a convex hypersurface with respect to $X$. Since $\alpha(X) = z$, the dividing set is $\Gamma = \{z=0\}\cap S^{2n}$ and the positive (resp. negative) region is the hemisphere $R_{\pm} = \{\pm z > 0\}\cap S^{2n}$. With $\phi = -z$, $(S^{2n}, \phi)$ is a Weinstein convex hypersurface. 
\end{example}

\begin{example}[Rounded contact handlebody]
More generally, if $(V,\lambda)$ is a Liouville domain, we may consider a compact \textit{contactization} $(V\times [-1,1]_t, \alpha = dt + \lambda)$. One may round the edges of the resulting contact manifold to obtain a smooth boundary $\Sigma$ which is naturally convex and has a dividing set contactomorphic to $\partial V$. See \cite{avdek2012liouville} for a careful description of the rounding.    
\end{example}

\subsection{Lefschetz fibrations}

Here we review the theory of Lefschetz fibrations. This is necessary for the statement and proof of \cref{thm:lefschetz}.

\begin{definition}\label{def:WLF}
A \textbf{Weinstein Lefschetz fibration} is a pair $(p:W^{2n} \to \D^2, \lambda)$ satisfying the following properties: 
\begin{enumerate}
    \item The manifold $W^{2n}$ is a compact domain with corners which admits a smoothing $W^{\mathrm{sm}}$ and a Morse function $\phi:W^{\mathrm{sm}} \to \R$ such that the \textbf{total space} $(W^{\mathrm{sm}}, \lambda, \phi)$ is a Weinstein domain.
    \item The map $p:W \to \D^2$ is a smooth fibration except at finitely many critical points in the interior of $W$ with distinct critical values, around which there are local holomorphic coordinates such that 
    \begin{align*}
        p(z) &= p(z_0) + \sum_{j=1}^n z_j^2, \\
        \lambda &= i\sum_{j=1}^n (z_j\, d\overline{z}_j -\overline{z}_j \,dz_j).
    \end{align*}
    \item The boundary $\partial W$ decomposes as  
    \[
    \big[\partial_{\mathrm{vert}} W := p^{-1}(\partial \D^2)\big] \, \cup \, \big[\partial_{\mathrm{hor}} W := \bigcup_{x\in \D^2} \partial ( p^{-1}(x))\big],
    \]
    where the two pieces meet at a codimension-$2$ corner, and $p|_{\partial_{\mathrm{vert}} W}$ and $p|_{\partial_{\mathrm{hor}} W}$ are fibrations.
    \item On each regular fiber, $\lambda$ induces a Weinstein structure such that the contact form $\lambda\mid_{\partial (p^{-1}(x))}$ is independent of $x\in \D^2$, and finally $d\lambda$ is nondegenerate on $\ker dp(z)$ for all $z\in W$ (not just regular points). 
\end{enumerate}
\end{definition}

Choose a regular basepoint $\bullet \in \D^2$ and let $W_0 := p^{-1}(\bullet)$ be a distinguished regular fiber. After isotoping the critical values $x_1, \dots, x_k$ to be approximately radially distributed around $\bullet$, we get a cyclically ordered tuple of Lagrangian spheres in $W_0$ by identifying Lefschetz thimbles over paths $\gamma_j$ that are approximately straight lines from $x_j$ to $\bullet$. This yields: 

\begin{definition}\label{def:AWLF}
An \textbf{abstract Weinstein Lefschetz fibration} is the data 
\[
\left((W_0^{2n-2}, \lambda_0, \phi_0);\, \mathcal{L} = (L_1, \dots, L_k)\right),
\]
abbreviated $(W_0; \mathcal{L})$, where $(W_0^{2n-2}, \lambda_0, \phi_0)$ is a Weinstein domain and $\mathcal{L}$ is an ordered $k$-tuple of exact parameterized Lagrangian ($n-1$)-spheres (possibly duplicated), called the \textbf{vanishing cycles}, embedded in $W_0$. The \textbf{total space of $(W_0; \mathcal{L})$}, denoted $|W_0; \mathcal{L}|$, is the $2n$-dimensional Weinstein domain obtained by attaching critical handles to $(W_0\times \D^2, \lambda_0 + \lambda_{\mathrm{st}}, \phi_0 + \phi_{\mathrm{st}})$ along Legendrian lifts $\Lambda_j\subseteq W_0\times \partial \D^2$, $j=1, \dots, k$, of $L_j$ positioned near $2\pi j/k \in \partial \D^2$.
\end{definition}

The total space $|W_0; \mathcal{L}|$ of an abstract Weinstein Lefschetz fibration naturally admits a Weinstein Lefschetz fibration $p:W\to \D^2$ with $W^{\mathrm{sm}}= |W_0; \mathcal{L}|$ in the sense of \cref{def:WLF}. We refer to Giroux and Pardon \cite[\S 6]{giroux2017lefschetz} for more details on translating between the two notions. 

Note that (3) of \cref{def:WLF} gives a \textit{supporting (strongly Weinstein) open book decomposition} of the contact manifold $(M, \xi):= \partial W^{\mathrm{sm}}$. A neighborhood of the binding is given by $\partial_{\mathrm{hor}}W$ while the partial mapping torus of the page is $\partial_{\mathrm{vert}}W$. The abstract description of the open book is easily obtained from the data of an abstract Weinstein Lefschetz fibration. Namely, given $(W_0; \mathcal{L})$, the induced abstract open book on the boundary has page $W_0$ and monodromy 
\[
\tau_{\mathcal{L}}:= \tau_{L_k} \circ \cdots \circ \tau_{L_1}
\]
where $\tau_{L_j}:W_0 \to W_0$ is the positive symplectic Dehn twist around the exact Lagrangian sphere $L_{j}\subset W_0$. 

Every Weinstein domain, up to homotopy, admits a description as the total space of a Weinstein Lefschetz fibration. 

\begin{theorem}[\cite{giroux2017lefschetz}]\label{thm:GP17}
Let $(W, \lambda, \phi)$ be a Weinstein domain. Then there is an abstract Weinstein Lefschetz fibration $(W_0; \mathcal{L})$ whose total space $|W_0; \mathcal{L}|$ is Weinstein homotopic to $(W,\lambda, \phi)$.     
\end{theorem}

There are various symmetries of Weinstein Lefschetz fibrations that leave the total space invariant. Most relevant for our purposes is a \textit{stabilization} operation.

\begin{definition}
Let $(W_0; \mathcal{L})$ be a Weinstein Lefschetz fibration. Let $D\subset W_0$ be a properly embedded regular Lagrangian disk. The \textbf{stabilization of $(W_0; \mathcal{L})$ along $D$} is $(W_0 \cup h^{2n-2}_{n-1}; L \cup \mathcal{L})$ where $h_{n-1}^{2n-2}$ is a critical Weinstein handle attached to $W_0$ along $\partial D$, and $L$ is the exact Lagrangian sphere which is the union of $D$ with the core of $h_{n-1}^{2n-2}$.    
\end{definition}

If $(W_0'; \mathcal{L}')$ is a stabilization of $(W_0; \mathcal{L})$, then the total space $|W_0'; \mathcal{L}'|$ is Weinstein homotopic to $|W_0; \mathcal{L}|$. Moreover, a Lefschetz fibration stabilization induces a positive stabilization of the boundary open book decomposition, and conversely, any positive stabilization of the boundary open book decompositon induces a stabilization of the Lefschetz fibration. 

\section{Exact folded symplectic forms}\label{sec:exact}

The role of this section is to provide the definitions needed to make our main results \cref{theorem:main0}, \cref{theorem:main}, and \cref{thm:lefschetz} precise. We first establish some language, constructions, and observations about exact folded forms, and then extend the discussion to folded Lefschetz fibrations. 

\subsection{Exact folded forms}

Many definitions from the setting of exact symplectic forms port directly over to the folded setting.

\begin{definition}
Let $\omega=d\lambda$ be an exact folded symplectic form on $\Sigma$ with fold $\Gamma$. The \textbf{Liouville vector field $X_{\lambda}$} is the unique vector field defined on $\Sigma \setminus \Gamma$ by the relation $X_{\lambda} \lrcorner \omega = \lambda$.    
\end{definition}

\begin{definition}\label{def:CTliouville}
Let $\omega$ be a an exact folded symplectic form on $\Sigma$ with fold $\Gamma$, and let $\iota:\Gamma \hookrightarrow \Sigma$ be the inclusion. Recall that $\Gamma$ is naturally oriented as $\partial\overline{R}_+$. We say that $\lambda\in \Omega^1(\Sigma)$ is a \textbf{positive contact-type Liouville form} for $\omega$ if $\omega = d\lambda$ and $\iota^*\lambda$ is a positive contact form on $\Gamma$. 
\end{definition}

\begin{example}\label{ex:darboux}
Consider the local Darboux model $\R^{2n}_{(x_1, y_1 ,\dots, x_n, y_n)}$ with fold $\Gamma = \{y_1=0\}$ and 
\[
\omega = y_1\, dx_1\wedge dy_1 + \sum_{j=2}^n dx_j \wedge dy_j.   
\]
Define two $1$-forms $\tilde{\lambda}$ and $\lambda$ as given to the left below: 
\begin{align*}
\tilde{\lambda} &:= -\frac{y_1^2}{2}\, dx_1 - \sum_{j=2}^n y_j\, dx_j, & X_{\tilde{\lambda}} &= \frac{1}{2}y_1\, \partial_{y_1} + \sum_{j=2}^n y_j\, \partial_{y_j},\\
\lambda &:= \left(1 -\frac{y_1^2}{2}\right)\, dx_1 - \sum_{j=2}^n y_j\, dx_j, & X_{\lambda} &= \left(\frac{y_1^2 - 2}{2y_1}\right)\, \partial_{y_1} + \sum_{j=2}^n y_j\, \partial_{y_j}.
\end{align*}
Note that $\omega = d\tilde{\lambda} = d\lambda$ and hence $\tilde{\lambda}$ and $\lambda$ are both Liouville forms for $\omega$. The Liouville vector fields are computed to the right above. However, $\tilde{\lambda}$ is not a positive contact-type Liouville form for $\omega$, while $\lambda$ is a positive contact-type Liouville form. Indeed, 
\begin{align*}
    \iota^*\tilde{\lambda} &= -\sum_{j=2}^n y_j\, dx_j, \\
\iota^*\lambda &= dx_1 - \sum_{j=2}^n y_j\, dx_j.
\end{align*}
The latter is a positive contact form on $\Gamma = \partial \{y_1\geq 0\}$, while the former is not a contact form at all. Observe that for the positive contact-type Liouville form $\lambda$, the Liouville vector field is dominated by $-\frac{1}{y_1}\, \partial_{y_1}$ near $\Gamma$; see \cref{fig:contacttype}. 
\end{example}

\begin{figure}[ht]
	\begin{overpic}[scale=0.3]{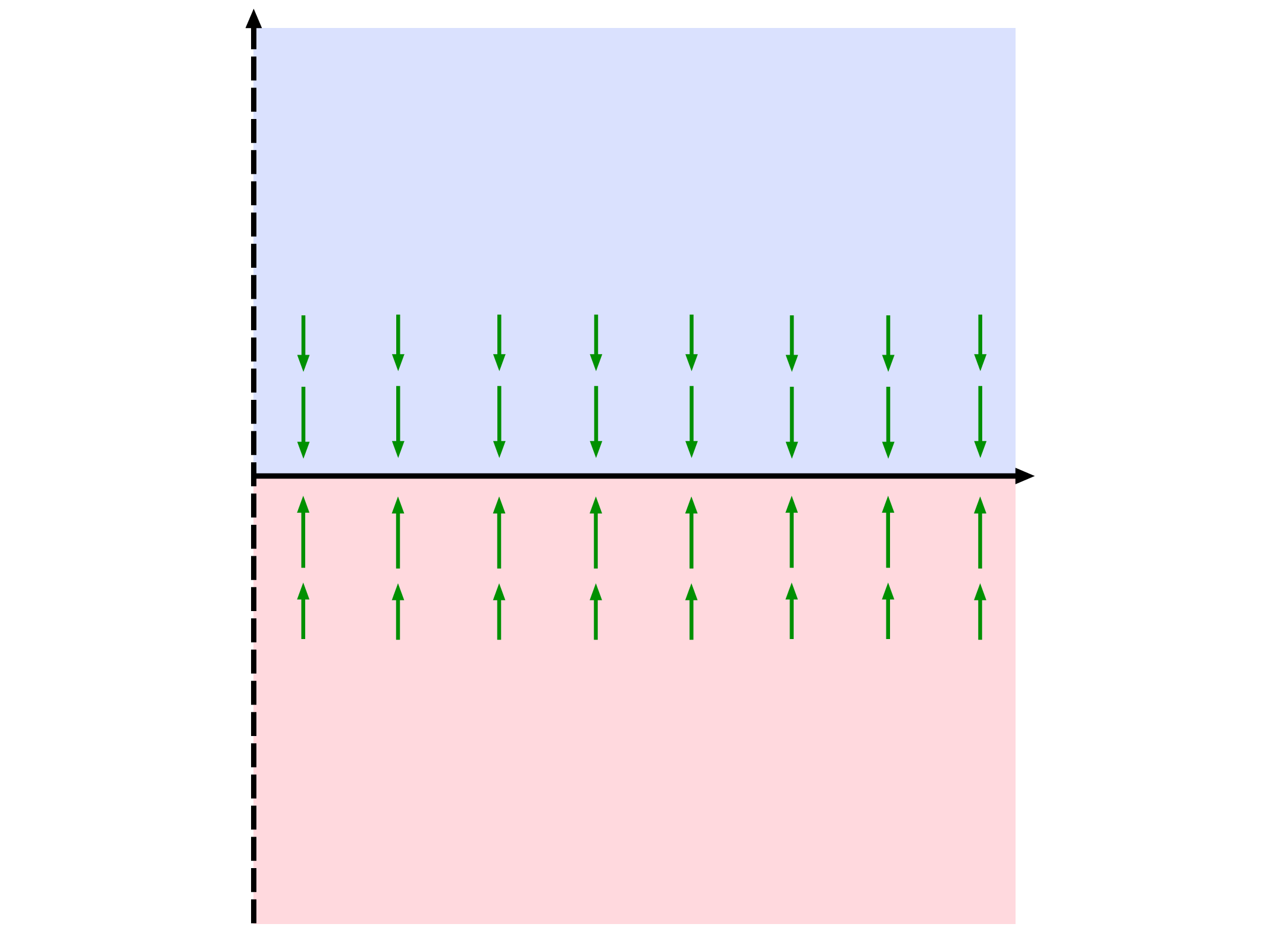}
     \put(83,37){\small $\Gamma$}
     \put(49,56){\small \textcolor{darkblue}{$R_+$}}
     \put(49,16){\small \textcolor{darkred}{$R_-$}}
     \put(65,21){\small \textcolor{darkgreen}{$X_{\lambda}$}}
     \put(18.5,75.5){\small $y_1$}
	\end{overpic}
	\caption{The local structure of a positive contact-type Liouville form near a fold.}
	\label{fig:contacttype}
\end{figure}
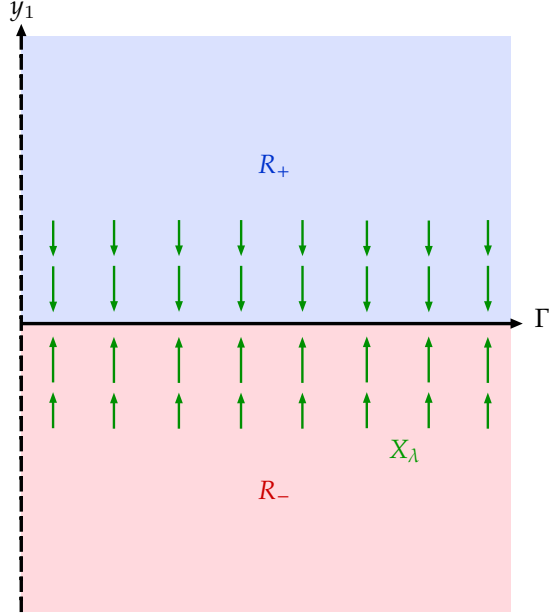

\begin{example}\label{ex:sphere}
Consider the standard folded symplectic structure on $S^{2n}$ given as follows (see, for instance, \cite[Section 3]{cannas2000unfolding}). We view $S^{2n} \subset \R^{2n+1}_{(x_1, y_1 ,\dots, x_n, y_n, z)}$ as the unit sphere. Let $\pi: S^{2n} \to \R^{2n}_{(x_1, y_1 ,\dots, x_n, y_n)}$ be the projection, and define $\omega\in \Omega^2(S^{2n})$ by $\omega := \pi^*\left(\sum_{j=1}^n dx_j \wedge dy_j\right)$. Then $\omega$ is a folded symplectic form with fold $\Gamma = S^{2n}\cap \{z=0\}$. Let 
\[
\lambda := \pi^*\left(\frac{1}{2} \sum_{j=1}^n x_j\, dy_j - y_j\, dx_j\right).
\]
Then $\omega = d\lambda$, and moreover $\lambda$ is a positive contact-type Liouville form for $\omega$ inducing the standard tight contact structure on $\Gamma \cong S^{2n-1}$. See for comparison the standard convex sphere from \cref{ex:convex_sphere}.
\end{example}

More generally, one can double a Liouville cobordism to obtain an exact folded symplectic manifold. 

\begin{example}[Double of a Liouville cobordism]\label{ex:double}
Let $(W^{2n}, d\lambda)$ be a Liouville cobordism with negative (concave) end $\partial_- W$ and positive (convex) end $\partial_+ W$. Let $\alpha_{\pm}$ be the induced contact form on $\partial_{\pm} W$. There are collar neighborhoods $N(\partial_{+} W)$ and $N(\partial_{-} W)$ such that 
\begin{align*}
    &N(\partial_+ W) \cong (-\ve, 0]_{s} \times \partial_+ W\, \text{ with }\, \lambda = e^s\, \alpha_+, \\
    &N(\partial_- W) \cong [0,\ve)_{s} \times \partial_- W\,\text{ with }\, \lambda = e^s\, \alpha_-.
\end{align*}
Let $\Sigma^{2n} \subset \R \times W$ be the smooth submanifold defined as follows: 
\begin{itemize}
    \item[$\diamond$] Away from $\R \times N(\partial_{\pm} W)$, set $\Sigma := (\{1\} \times W) \, \cup\,  (\{-1\}\times -W)$. 
    \item[$\diamond$] In $\R_z \times N(\partial_+ W)$, let $\Sigma:=\{s = f_+(z)\}$ where $f_+:(-1,1)_z \to (-\ve, 0]_s$ is a smooth function satisfying $f_+(0) = f_+'(0) = 0$, $f_+''(z)<0$, and $f_+(z) \to -\ve$ and $f'_+(z) \to \pm\infty$ as $z \to \mp 1$.
    \item[$\diamond$] Likewise, in $\R_z \times N(\partial_- W)$ let $\Sigma:=\{s=f_-(z)\}$ where $f_-: (-1,1)_z \to [0,\ve)_s$ is a smooth function satisfying $f_-(0) = f'_-(0) = 0$, $f_-''(z) > 0$, and $f_-(z) \to \ve$ and $f'_-(z) \to \pm \infty$ as $z \to \pm 1$. 
\end{itemize}
See the left side of \cref{fig:double}. Let $\pi:\Sigma \subset \R \times W \to W$ be the projection and define $\omega := \pi^*(d\lambda)$. Then $\omega$ is an exact folded symplectic form on $\Sigma$ with fold $\Gamma = \Sigma \cap \{z=0\}$. Each of $R_+$ and $R_-$ are essentially copies of the interior of the original Liouville cobordism. 

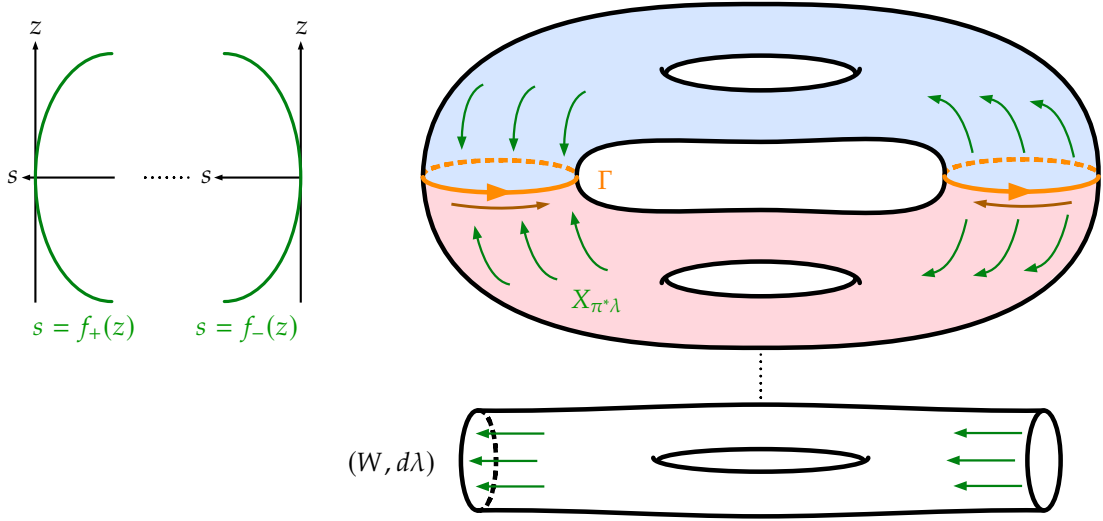
\begin{figure}[ht]
	\begin{overpic}[scale=0.4]{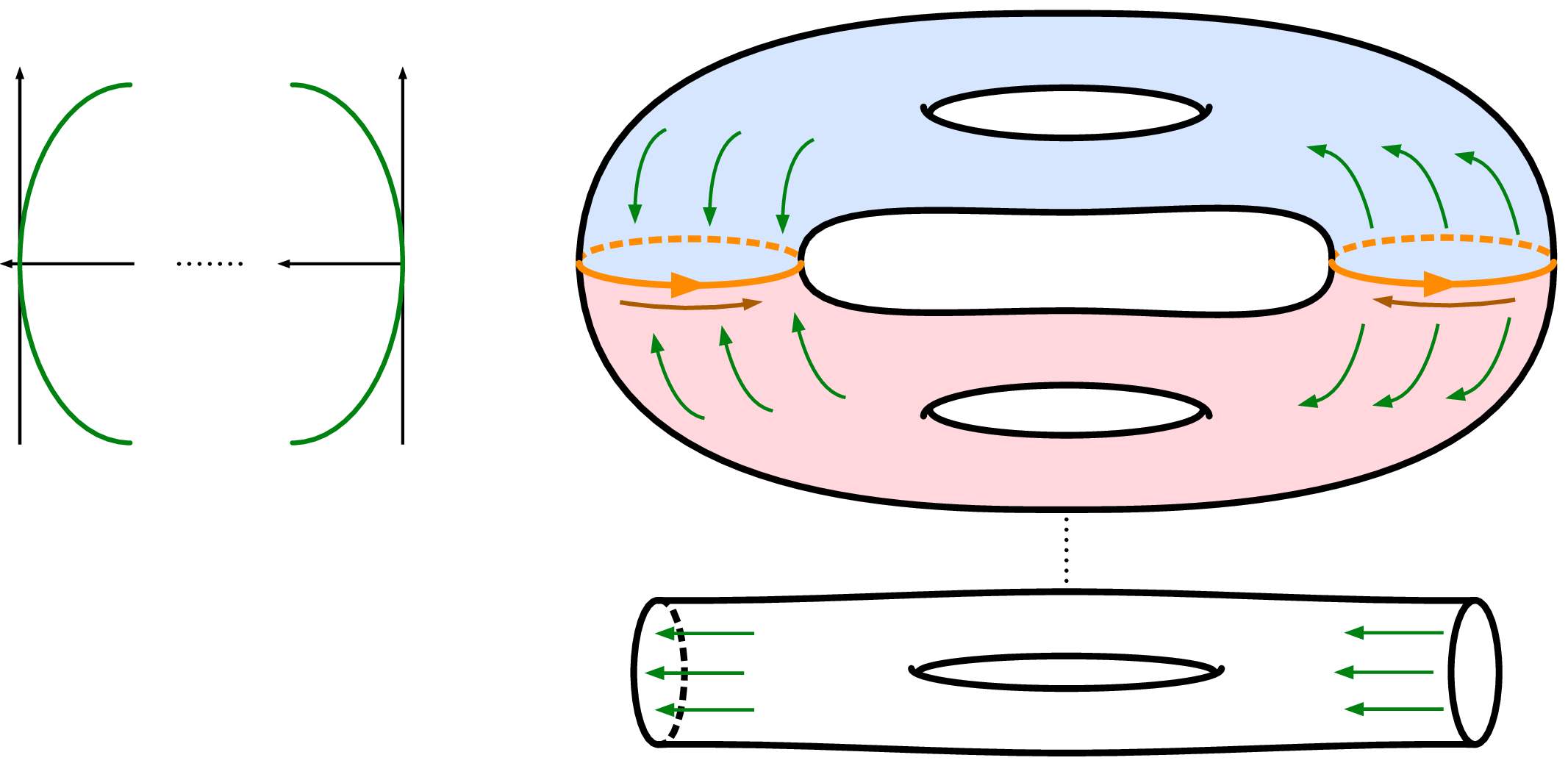}
     \put(53,31){\small \textcolor{orange}{$\Gamma$}}
     \put(50.5,20){\small \textcolor{darkgreen}{$X_{\pi^*\lambda}$}}
     \put(30,5){\small $(W,d\lambda)$}

     \put(0.75,45.25){\small $z$}
     \put(1,17.5){\small \textcolor{darkgreen}{$s=f_+(z)$}}
     \put(16,17.5){\small \textcolor{darkgreen}{$s=f_-(z)$}}
     \put(25.2,45.25){\small $z$}

     \put(-1.3,31.5){\small $s$}
     \put(16.5,31.5){\small $s$}
	\end{overpic}
	\caption{The double of a Liouville cobordism. The fold is oriented as the boundary of $\overline{R}_+$; this is depicted by the orange direction. The brown direction depicts the orientation induced by the contact form. Compare the local depiction of the Liouville vector fields with \cref{fig:contacttype}.}
	\label{fig:double}
\end{figure}

This example indicates the importance of the adjective \textit{positive} when talking about contact-type Liouville forms. While $\pi^*\lambda$ is a contact form on the fold $\Gamma$, if the original Liouville cobordism has a nonempty negative end $\partial_- W$, then $\pi^* \lambda$ will not be a \textit{positive} contact-type Liouville form. Indeed, the fold $\Gamma$ is oriented as the boundary of $\overline{R}_+$, but the contact form on the negative boundary is induced as the negative end of a Liouville cobordism and hence on this component of the fold the contact form will not be positive. See the right side of \cref{fig:double}.
\end{example}

\begin{definition}\label{def:folded_weinstein}
A \textbf{folded Weinstein structure} on $\Sigma$ is the data $(\lambda, \phi)$ where 
\begin{enumerate}
    \item $\omega:= d\lambda$ is an exact folded symplectic form with fold $\Gamma$ and positive (resp. negative) region $R_{\pm}$,
    \item $\lambda$ is a positive contact-type Liouville form for $\omega$, and 
    \item $\phi:\Sigma \to \R$ is a (generalized) Morse function such that $\Gamma = \phi^{-1}(0)$ is a regular level set and such that $X_{\lambda}\mid_{R_{\pm}}$ is gradient-like for $\pm \phi\mid_{R_{\pm}}$.
\end{enumerate}
The triple $(\Sigma, \lambda, \phi)$ is a \textbf{folded Weinstein manifold}. A \textbf{folded Weinstein homotopy} is simply a $1$-parameter family $(\Sigma, \lambda_t, \phi_t)$, $t\in [0,1]$, of folded Weinstein structures on $\Sigma$. 
\end{definition}

\begin{remark}\label{remark:retract_domain}
Let $\ve > 0$ be sufficiently small and let
\begin{align*}
  R_+^{\ve} &:= \phi^{-1}(-\infty,-\ve], \\
    R_-^{\ve} &:= \phi^{-1}[\ve,\infty) 
\end{align*}
be slight compact retracts of $R_{+}$ and $R_-$. Then $(R_+^{\ve}, \lambda, \phi)$ and $(-R_-^{\ve}, \lambda, -\phi)$ are both Weinstein domains. It is not correct to call $(\pm\overline{R}_{\pm},\lambda,\pm \phi)$ a Weinstein domain, as $\omega=d\lambda$ does not extend through the fold to a symplectic form. One could also ask if the open manifold $(\pm R_{\pm}, \lambda, \pm \phi)$ is a completed Weinstein manifold, or if its closure is an ideal Liouville domain, but this is also not the case --- note that (refer back to the local model of \cref{ex:darboux}) the Liouville vector field is not complete.    
\end{remark}

\subsubsection{Asymmetric double of two Weinstein domains}

The double of a Weinstein domain as in \cref{ex:double} is naturally a folded Weinstein manifold. More generally, we can form an asymmetric double of two Weinstein domains with contactomorphic boundaries. For the sake of precision, we give a careful description of the construction. See also \cite[Section 6]{cannas2000unfolding}.

Let $(W_+, \lambda_+, \phi_+)$ and $(W_-, \lambda_-, \phi_-)$ be two compact Weinstein domains of the same dimension. Assume moreover that the contact manifolds $(\partial W_{+}, \ker \alpha_{+}) \cong (\partial W_{-}, \ker \alpha_{-})$ are contactomorphic, where $\alpha_{\pm} = \iota_{\pm}^{*}\lambda_{\pm}$ is the induced contact form under the inclusion $\iota_{\pm}: \partial W_{\pm} \hookrightarrow W_{\pm}$. 

Let $\psi: (\partial W_-,\alpha_-) \to (\partial W_+, \alpha_+)$ be a contactomorphism, so that $\psi^*\alpha_+ = \mu\, \alpha_-$ for some $\mu: \partial W_- \to \R_{>0}$. We first wish to extend this contactomorphism to a Liouville isomorphism of collar neighborhoods of the boundaries. Using the Liouville flow on $W_-$, we may identify a compact collar neighborhood $N(\partial W_-) \cong [-\ve, 0]_s \times \partial W_-$ on which $\lambda_- = e^s\, \alpha_-$. On $W_+$, we first take the completion $\hat{W}_+$, let $C>\sup(\ve + |\ln \mu|)$ be a constant, and then identify a collar neighborhood $N(\partial W_+) \cong [-C, C]_s \times \partial W_+$ of the original boundary $\partial W_+ \subset \hat{W}_+$ on which $\lambda_+ = e^s\, \alpha_+$. Then define $\bar{\psi}:N(\partial W_-) \to N(\partial W_+)$ by 
\[
\bar{\psi}(s,p) := \left(s - \ln \mu, \, \psi(p)\right).
\]
We have 
\[
\bar{\psi}^*\lambda_+ = \bar{\psi}^*(e^s\, \alpha_+) = e^{s-\ln \mu}\, \psi^*\alpha_+ = e^s\cdot \frac{1}{\mu}\cdot \mu \cdot \alpha_- = e^s\, \alpha_- = \lambda_-. 
\]
and so $\bar{\psi}$ is a Liouville isomorphism onto its image. In particular, $\bar{\psi}$ identifies a collar neighborhood of $\partial W_-\subset W_-$ with a collar neighborhood of the Weinstein subdomain $W^{\flat}_+\subset \hat{W}_+$ of the completion of $W_+$ given by $W^{\flat}_+ := \hat{W}_+ \setminus \{s > -\ln \mu\}$. Here we endow $W_+^{\flat}$ with a Morse function $\phi_+^{\flat}: W_+^{\flat} \to \R$ that agrees with $\phi_+$ on $W_+ \setminus \{s> -C\}$, satisfies $d\phi(\partial_s) > 0$, and has $\partial W_+^{\flat}$ as a regular level set. Note that the domain $W^{\flat}_+$ is Weinstein homotopic to $W_+$. 

Next we mimic \cref{ex:double} with a few modifications. Let $f:(-1,1)_z \to [0,1]_s$ be a smooth function satisfying $f(0) = f'(0) = 1$, $f''(z) < 0$, and $f(z) \to 0$ and $f'(z) \to \pm \infty$ as $z \to \mp 1$. Consider the manifold $M_0 := \R_z \times [-\ve, 1]_s \times \partial W_-$ and define a submanifold $\Sigma_0\subset M_0$ by 
\[
\Sigma_0\,\, := \,\,\{s = f(z)\} \,\, \cup \,\, (\{1\}_z \times [-\ve, 0]_s \times \partial W_-)\,\, \cup \,\,(\{-1\}_z \times [-\ve, 0]_s \times \partial W_-).
\]
Define a closed manifold $\Sigma$ by $\Sigma:= \Sigma_0 \cup_{\mathrm{id}} W_- \cup_{\bar{\psi}} W_+^{\flat}$, where we glue $\Sigma_0$ and $W_-$ via the obvious map
\[
\mathrm{id}: \{-1\}_z \times [-\ve, 0]_s \times \partial W_- \to N(\partial W_-) \cong [-\ve, 0]_s \times \partial W_-
\]
and we glue $\Sigma_0$ and $W_+^{\flat}$ instead by the corresponding map 
\[
\bar{\psi}: \{1\}_z \times [-\ve, 0]_s \times \partial W_- \to \bar{\psi}(N(\partial W_-)) \subset W_+^{\flat}
\]
induced by the original contactomorphism. See \cref{fig:asym}. Note that $\Sigma$ is naturally an oriented manifold such that $W_+^{\flat}$ and $-W_-$ are naturally oriented submanifolds. That is, the orientation of $\Sigma$ agrees with the orientation of $W_+^{\flat}$ induced by $d\lambda_+$ and disagrees with the orientation of $W_-$ induced by $d\lambda_-$. 

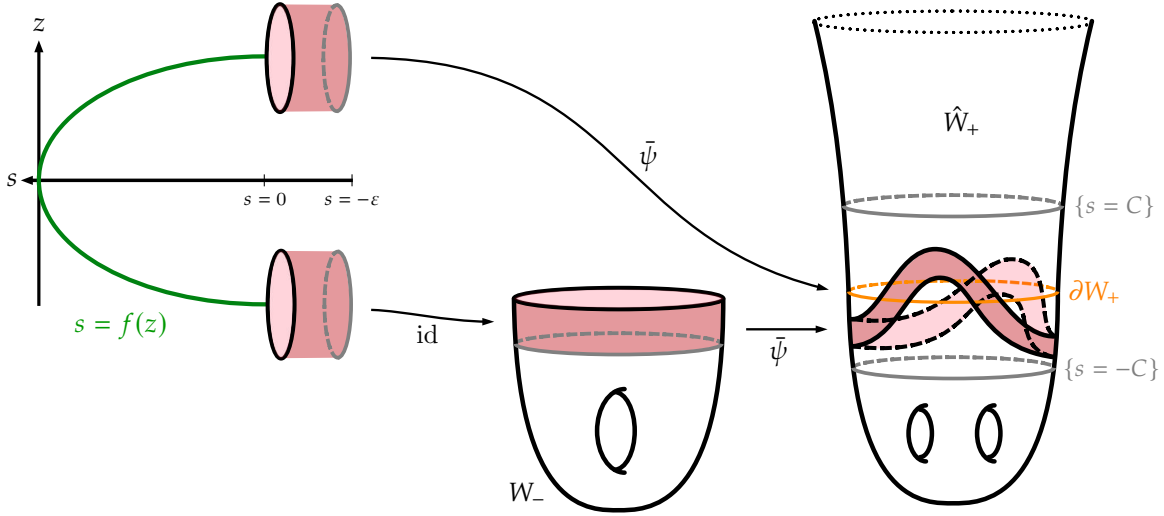
\begin{figure}[ht]
	\begin{overpic}[scale=0.4]{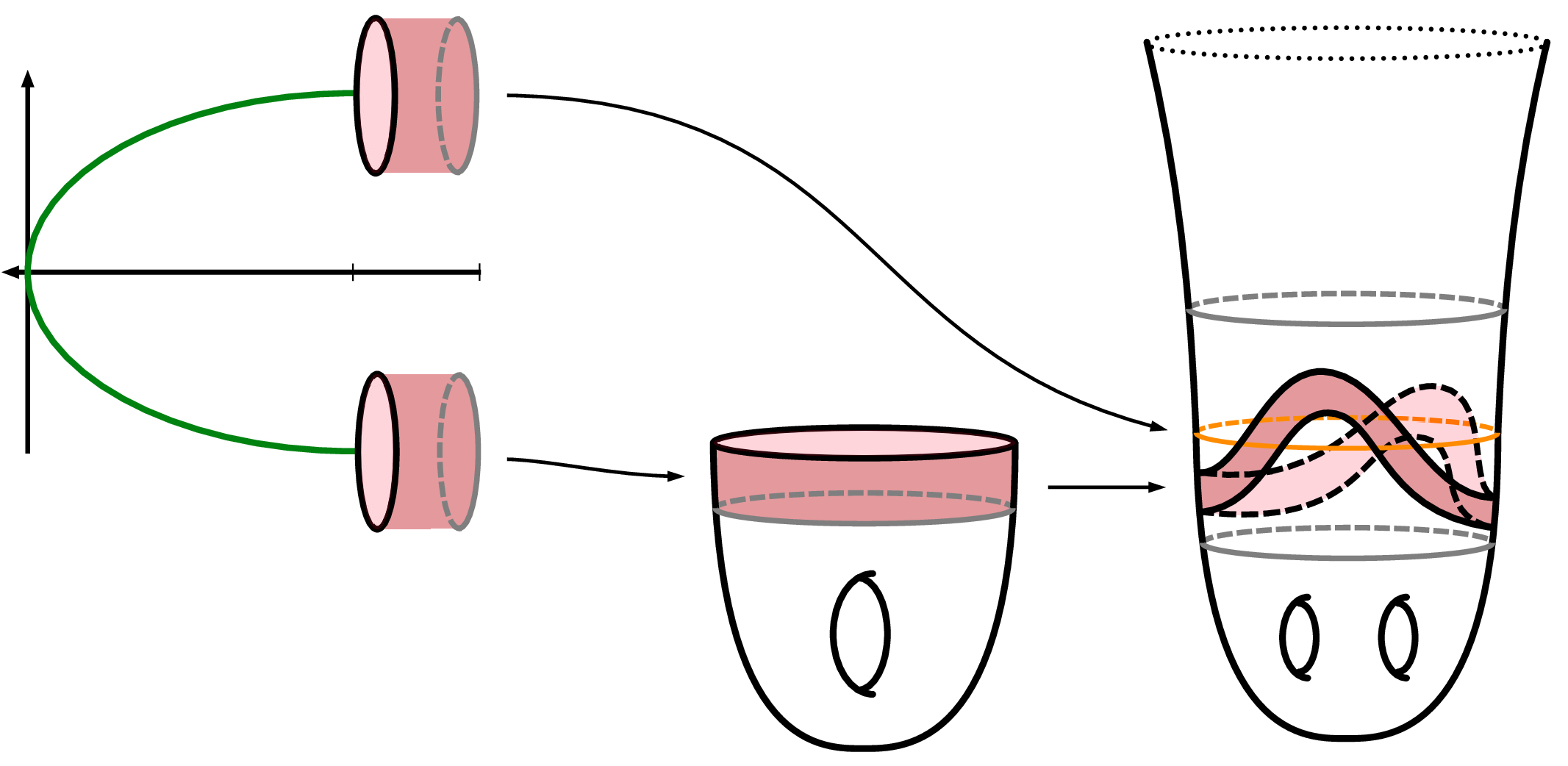}
     \put(96.5,20.5){\small \textcolor{orange}{$\partial W_+$}}
    \put(97,28.5){\footnotesize \textcolor{gray}{$\{s=C\}$}}
    \put(96,13.5){\footnotesize \textcolor{gray}{$\{s=-C\}$}}
     
     \put(45,2){\small $W_-$}
     \put(85,36){\small $\hat{W}_+$}

     \put(1.2,45.25){\small $z$}
    \put(-1.2,31){\small $s$}

    \put(20.6,29.25){\tiny $s=0$}
    \put(28,29.25){\tiny $s=-\ve$}
     
     \put(5,17.5){\small \textcolor{darkgreen}{$s=f(z)$}}

    \put(36.5,16.5){\small $\mathrm{id}$}
    \put(69,15){\small $\bar{\psi}$}
    \put(57,33){\small $\bar{\psi}$}
     
	\end{overpic}
	\caption{Constructing the asymmetric double of two Weinstein domains.}
	\label{fig:asym}
\end{figure}

Next we define the folded symplectic structure. Consider the $1$-form $e^s\, \alpha_-$ on $[-\ve, 1]_s \times \partial W_-$. Let $\pi: M_0 \to [-\ve, 1]_s \times \partial W_-$ be the projection and define $\lambda_0$ on $\Sigma_0$ by $\lambda_0 := \pi^*(e^s\, \alpha_-)$. By construction, we may extend $\lambda_0$ to a $1$-form on all of $\Sigma$ via $\lambda_+$ on $W^{\flat}_+$ and $\lambda_-$ on $W_-$, both of which agree with $\lambda_0$ on their respective gluing regions. By the orientation remark above, it follows that $d\lambda$ is a folded symplectic form on $\Sigma$ with fold $\Gamma = \{z=0\} \cap \Sigma_0$. Moreover, $\lambda$ is a positive contact-type Liouville form inducing on the fold the same contact structure as the boundary of both initial Weinstein domains. Indeed, the induced $1$-form on the fold is $e^{1}\, \alpha_-$.

Finally, we upgrade this exact folded symplectic structure to a folded Weinstein structure in the natural way. By shifting the  Morse functions $\phi_+^{\flat}:W_+^{\flat} \to \R$ and $\phi_-:W_- \to \R$ by constants, we may assume that $\sup \phi_+^{\flat} = -1$ and $\inf -\phi_- = 1$.  Define $\phi: \Sigma \to \R$ by $\phi:= \phi_+^{\flat}$ on $W_+^{\flat}$, $\phi := -\phi_-$ on $W_-$, and by $\phi := -z$ on $\Sigma_0 \cap \{s\geq 0\}$. Then by construction, $(\Sigma, \lambda, \phi)$ is a folded Weinstein manifold. 

\begin{definition}\label{def:asym_double}
Let $(W_+, \lambda_+, \phi_+)$ and $(W_-, \lambda_-, \phi_-)$ be two compact Weinstein domains of the same dimension, and let $\psi: (\partial W_-, \xi_-) \to (\partial W_+, \xi_+)$ be a contactomorphism of the boundaries. Define the \textbf{asymmetric double of $W_+$ and $W_-$ along $\psi$}, denoted $D(W_+, W_-, \psi)$, to be the folded Weinstein manifold $(\Sigma, \lambda, \phi)$ obtained as the result of the above construction.
\end{definition}

\subsection{Folded Lefschetz fibrations}

The following definition is the folded analogue of \cref{def:AWLF}.

\begin{definition}\label{def:AFWLF}
An \textbf{(abstract) folded Weinstein Lefschetz fibration} is the data 
\[
\left((W_0^{2n-2}, \lambda_0, \phi_0);\,\, \mathcal{L}^+ = (L_1^+, \dots, L_{k_+}^+),\, \mathcal{L}^- = (L_1^-, \dots, L_{k_-}^-) \right),
\]
abbreviated $(W_0; \mathcal{L}^+, \mathcal{L}^-)$, where $(W_0^{2n-2}, \lambda_0, \phi_0)$ is a Weinstein domain and both $(W_0; \mathcal{L}^+)$ and $(W_0; \mathcal{L}^-)$ are abstract Weinstein Lefschetz fibrations satisfying 
\[
\tau_{L_{k_+}^+} \circ \cdots \circ \tau_{L_{1}^+} = \tau_{L_{k_-}^-} \circ \cdots \circ \tau_{L_{1}^-}  
\]
where $\tau_{L_j^{\pm}}: W_0 \to W_0$ is the positive symplectic Dehn twist around $L_{j}^{\pm}$ and equality is in the sense of symplectic isotopy. The \textbf{total space of $(W_0; \mathcal{L}^+, \mathcal{L}^-)$}, denoted $|W_0; \mathcal{L}^+, \mathcal{L}^-|$, is the folded Weinstein manifold $\Sigma$ obtained by taking the asymmetric double of $|W_0; \mathcal{L}^+|$ and $|W_0; \mathcal{L}^-|$, i.e., in the notation of \cref{def:asym_double}, $\Sigma := D(|W_0; \mathcal{L}^+|, |W_0; \mathcal{L}^-|, \mathrm{id})$.
\end{definition}

\begin{definition}\label{def:suppAFWLF}
Let $(\Sigma, \lambda, \phi)$ be a folded Weinstein manifold. We say that $(\Sigma, \lambda, \phi)$ is \textbf{supported by a folded Weinstein Lefschetz fibration $(W_0; \mathcal{L}^+, \mathcal{L}^-)$} if $|W_0; \mathcal{L}^+, \mathcal{L}^-|$ is folded Weinstein homotopic to $(\Sigma, \lambda, \phi)$.   
\end{definition}

\section{Proofs of main results}\label{sec:proofs}

There are three subsections. \cref{subsec:main0} contains the proof of \cref{theorem:main0}, i.e., the correspondence between contact germs and exact folded symplectic forms with a positive contact-type Liouville form. \cref{subsec:main} contains the proof of \cref{theorem:main}, i.e., the upgraded folded Weinstein existence result. We also include a discussion on how our proof compares with the method of Cannas da Silva in \cite{cannas2010foldedfour}. Finally, \cref{subsec:lefschetzproof} contains the proof of \cref{thm:lefschetz}, i.e., the existence of supporting folded Weinstein Lefschetz fibrations. 

\subsection{Proof of \cref{theorem:main0}}\label{subsec:main0} We prove the equivalence of (1) and (2) via two lemmas. The first lemma, \cref{lemma:FoldedtoCHT}, describes how an exact folded symplectic structure on $\Sigma$ determines the germ of a contact structure along $\Sigma$, assuming a positive contact-type Liouville form. This proves that (1) implies (2). The second lemma, \cref{lemma:CHTtoFolded}, gives the converse direction, namely, that a convex hypersurface naturally acquires an exact folded symplectic form with positive contact-type Liouville form. This completes the proof of \cref{theorem:main0}.

\begin{lemma}[Exact folded $\rightsquigarrow$ contact germ]\label{lemma:FoldedtoCHT}
Let $\Sigma$ be a closed and oriented manifold of dimension $2n\geq 2$, and let $\omega = d\lambda$ be an exact folded symplectic form with positive contact-type Liouville form $\lambda$ and contact fold $(\Gamma, \xi_{\Gamma})$. Then there exists a smooth function $f:\Sigma \to \R$ such that $\{f=0\} = \Gamma$ and such that $\alpha := f\, dt + \lambda$ is a vertically-invariant contact form on $\Sigma \times \R_t$. In particular, $\Sigma = \Sigma \times \{0\}$ is a convex hypersurface with contact dividing set $(\Gamma, \xi_{\Gamma})$.
\end{lemma}

\begin{proof}
As usual, let $\iota:\Gamma \hookrightarrow \Sigma$ denote the inclusion of the fold into $\Sigma$. Let $\lambda_{\Gamma} := \iota^*\lambda$ and likewise let $\omega_{\Gamma} := \iota^*\omega = d\lambda_{\Gamma}$. For any function $f:\Sigma \to \R$ and $\alpha = f\, dt + \lambda$, we have 
\begin{equation}
    \alpha \wedge (d\alpha)^n = dt\wedge \left(f\, \omega^n - n\, df \wedge \lambda \wedge \omega^{n-1}\right).
\end{equation}
Thus, we need to show that $\Omega_f := f\, \omega^n - n\, df \wedge \lambda \wedge \omega^{n-1}$ is an orienting volume form on $\Sigma$ for some choice of $f$. 

We first localize near the fold. Since $\lambda$ is a contact-type Liouville form, the null-foliation $\ker \omega_{\Gamma}$ is the Reeb direction of $\lambda_{\Gamma}$. Therefore, by \cref{lemma:fold_normal}, we may identify a neighborhood $N(\Gamma) \cong \Gamma \times (-\ve, \ve)_{\tau}$ of the fold in $\Sigma$ such that\footnote{We remind the reader of \cref{remark:orientation}. That is, $\Sigma$ is oriented so that a volume form on $N(\Gamma)$ is $\Omega_{\Gamma} \wedge d\tau$ for an orienting volume form $\Omega_{\Gamma}$ on $\Gamma$. In particular, $R_{+} \cap N(\Gamma) = \{\tau > 0\}$.} 
\begin{align*}
    \omega &= \omega_{\Gamma} - d(\tau^2\, \lambda_{\Gamma}) \\
    &= d\left[(1 - \tau^2)\, \lambda_{\Gamma}\right] \\
    &= -2\tau \, d\tau \wedge \lambda_{\Gamma} + (1-\tau^2)\, \omega_{\Gamma}.
\end{align*}
In particular, on $N(\Gamma)$ we have $\lambda = (1 - \tau^2)\, \lambda_{\Gamma}.$ Without loss of generality we may assume $\ve < 1$. 

Now we compute, with the end goal of $\Omega_f$ in mind. We have
\begin{align*}
    \omega^{n-1} &= 2(n-1)\tau(1-\tau^2)^{n-2}\, \lambda_{\Gamma}\wedge \omega_{\Gamma}^{n-2}\wedge d\tau\, +\, (1-\tau^2)^{n-1}\, \omega_{\Gamma}^{n-1} \\
    \omega^{n} &= 2n\tau(1-\tau^2)^{n-1}\,\lambda_{\Gamma}\wedge \omega_{\Gamma}^{n-1}\wedge d\tau \\
    \lambda\wedge \omega^{n-1} &= (1-\tau^2)^n\, \lambda_{\Gamma}\wedge \omega_{\Gamma}^{n-1}
\end{align*}
and hence, assuming $f=f(\tau)$ on $N(\Gamma)$,  
\begin{align*}
    \Omega_f &= f\, \omega^n - n\, df \wedge \lambda \wedge \omega^{n-1} \\
    &= n(1-\tau^2)^{n-1}\, \left[2\tau f(\tau) + f'(\tau)(1-\tau^2)\right]\, \lambda_{\Gamma}\wedge \omega_{\Gamma}^{n-1} \wedge d\tau.
\end{align*}
Since $\lambda$ is a positive contact-type Liouville form, $\lambda_{\Gamma}\wedge \omega_{\Gamma}^{n-1} \wedge d\tau > 0$ on $N(\Gamma)$. Hence it suffices to choose a function $f(\tau)$ such that $2\tau f(\tau) + f'(\tau)(1-\tau^2)>0$.

Let $f:[-\ve,\ve]_{\tau} \to \R$ be a smooth function satisfying $f(\pm \ve) = \pm 1$, $f'(\pm \ve) = 0$, $f'(\tau) >0$ for $-\ve < \tau < \ve$, and $f(0) = 0$; see \cref{fig:function}. Then $2\tau f(\tau) > 0$ for $\tau \neq 0$ and $f'(\tau)(1-\tau^2) > 0$ for $-\ve <\tau < \ve$, hence $2\tau f(\tau) + f'(\tau)(1-\tau^2)>0$ for all $-\ve \leq \tau \leq \ve$. 

\begin{figure}[ht]
	\begin{overpic}[scale=0.3]{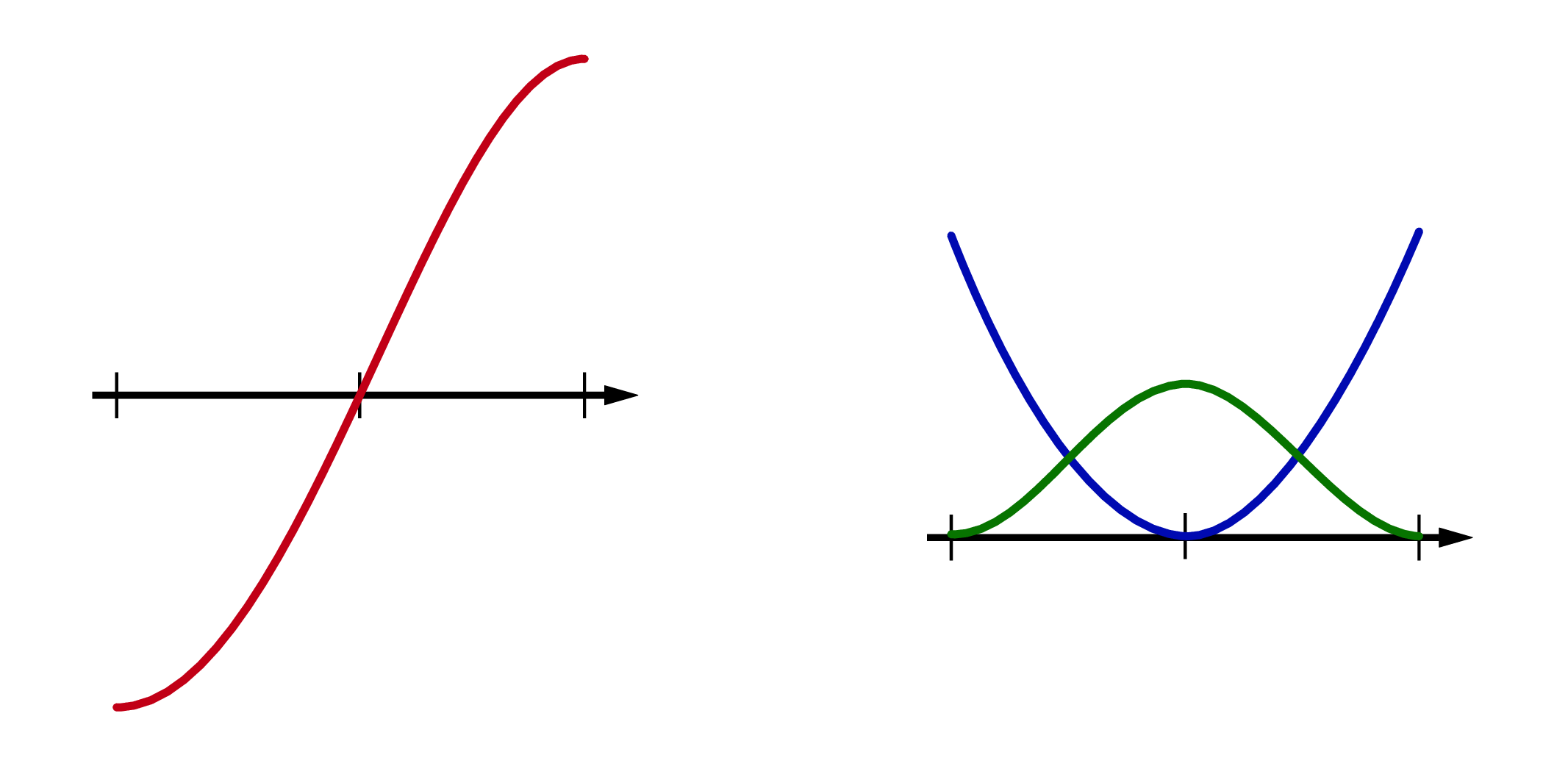}
     \put(41.25,23){\small $\tau$}
     \put(94.5,14){\small $\tau$}
     \put(86.5,21){\small \textcolor{darkblue}{$2\tau f(\tau)$}}
     \put(29,32){\small \textcolor{darkred}{$f(\tau)$}}
     \put(67.5,27){\small \textcolor{darkgreen}{$f'(\tau)(1-\tau^2)$}}

    \put(36.5,20){\small $\ve$}
    \put(5,20){\small $-\ve$}

    \put(89.5,11){\small $\ve$}
    \put(58,11){\small $-\ve$}
     
	\end{overpic}
    \vskip-0.43cm
	\caption{The function $f(\tau)$ used in the proof of \cref{lemma:FoldedtoCHT}.}
	\label{fig:function}
\end{figure}
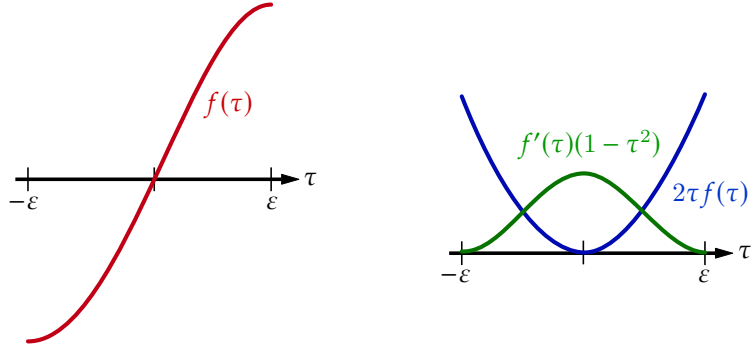

With $f$ defined appropriately on (the closure of) $N(\Gamma)$, it remains to extend $f$ across $\Sigma \setminus N(\Gamma)$. Thanks to the condition $f'(\pm \ve) = 0$, we can do so smoothly by defining $f \equiv \pm 1$ on $R_{\pm} \setminus N(\Gamma)$. Then on $\Sigma \setminus N(\Gamma)$ we have $df=0$ and hence on this region
\[
\Omega_f = \mathrm{sgn}(f)\, \omega^n.
\]
Since $\omega$ is folded symplectic, $\Omega_f > 0$. This completes the proof. 
\end{proof}

\begin{lemma}[Contact germ $\rightsquigarrow$ exact folded]\label{lemma:CHTtoFolded}
Let $\Sigma$ be a closed and orientable manifold of dimension $2n\geq 2$, and let $\alpha = f\, dt + \beta$ be a contact form on $\Sigma \times \R_t$ where $f: \Sigma \to \R$ and $\beta \in \Omega^1(\Sigma)$ are $t$-independent. Let $(\Gamma =\{f=0\}, \xi_{\Gamma})$ be the contact dividing set of $\Sigma = \Sigma \times \{0\}$. Then there is a smooth function $g:\Sigma \to \R_{>0}$ such that $\omega := d(g\, \beta)$ is an exact folded symplectic form with positive contact-type Liouville form and contact fold $(\Gamma, \xi_{\Gamma})$.  
\end{lemma}

\begin{proof}
The proof is similar to the proof of \cref{lemma:FoldedtoCHT}, appealing instead to the normal form near the dividing set of a convex hypersurface. 

\vspace{2mm}
\noindent \emph{Step 1: Normalize $f$.}
\vspace{2mm}

First, we rescale $\alpha = f\, dt + \beta$ by a positive function $\Sigma \to \R_{>0}$ to produce a different $t$-invariant contact form for the same contact structure. To describe the normalization we localize near the dividing set. By \cref{lemma:dividingsetnormal}, near $\Gamma$ there are coordinates on a neighborhood $N(\Gamma) \cong \Gamma \times (-\ve, \ve)_{\tau}$ of $\Gamma$ in $\Sigma$ such $f = \tau$ and $\beta = \beta_{\Gamma}$, where $\beta_{\Gamma} = \iota^*\beta$ is a contact form on $\Gamma$. Fix $0 < \delta < \ve$ and define a smaller collar neighborhood $N^{\delta}(\Gamma) := \Gamma \times (-\delta, \delta)_{\tau}$. Next we construct a smooth function $\mu: \Sigma \to \R_{>0}$ as follows: 
\begin{itemize}
    \item[$\diamond$] On $\Sigma \setminus N(\Gamma)$, set $\mu = |f|$. 
    \item[$\diamond$] On $N(\Gamma)$, let $\mu=\mu(\tau)$ be a smooth function that agrees with $|\tau|$ near $\tau = \pm \ve$, is a constant value $\mu \equiv \ve'$ for some $\delta \ll \ve' < \ve$ on $N^{\delta}(\Gamma)$, satisfies $0\leq \mu'(\tau) \leq 1$ for $\delta < \tau < \ve$, and satisfies $-1\leq \mu'(\tau) \leq 0$ for $-\ve < \tau < -\delta$; see \cref{fig:function2}.
\end{itemize}
Set $\tilde{\alpha} := \tilde{f}\, dt + \tilde{\beta}$ where $\tilde{f} := \frac{1}{\mu}f$ and $\tilde{\beta} := \frac{1}{\mu} \, \beta$. Then $\tilde{f}:\Sigma \to \R$ and $\tilde{\beta}$ satisfy the following properties: 
\begin{enumerate}
    \item On $R_{\pm} \setminus N(\Gamma)$, we have $d\tilde{f} = 0$ and $\pm (d\tilde{\beta})^n > 0$.
    \item On $N^{\delta}(\Gamma) = \Gamma \times (-\delta, \delta)_{\tau}$, we have $\tilde{f} = C\tau$ for some constant $C > 0$ and $\tilde{\beta} = \tilde{\beta}_{\Gamma} := \iota^*\tilde{\beta}$ where $\tilde{\beta}_{\Gamma}$ is a contact form on $\Gamma$. 
    \item On $N(\Gamma)$ we have $\tilde{f}'(\tau) \geq 0$. 
\end{enumerate}
Property (1) follows from \cref{lemma:contactnbd}. Indeed, since $d\tilde{f} = 0$, $\tilde{f}\, (d\tilde{\beta})^n$ is an orienting volume form on $R_{\pm} \setminus N(\Gamma)$. Property (2) is immediate by construction. To see (3), note that on $N(\Gamma)$ we have $\tilde{f}'(\tau) = \mu(\tau)^{-2}(\mu(\tau) - \tau\mu'(\tau))$. By the construction of $\mu$, we have $\mu(\tau) \geq |\tau|$ and $|\tau \mu'(\tau)| \leq |\tau|$, so 
\[
\mu(\tau) - \tau\mu'(\tau) \geq \mu(\tau) - |\tau|\geq 0
\]
which implies $\tilde{f}'(\tau) \geq 0$. We now rename $f=\tilde{f}$ and $\beta = \tilde{\beta}$. 

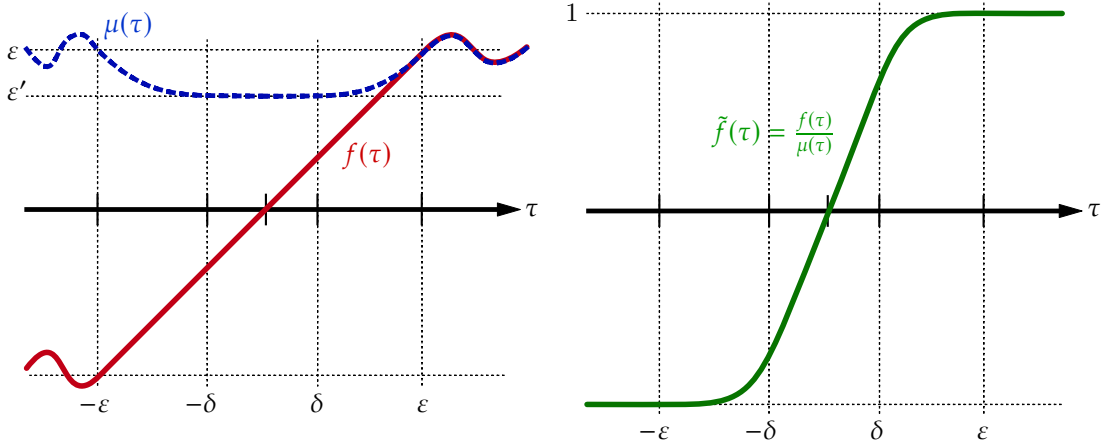
\begin{figure}[ht]
	\begin{overpic}[scale=0.4]{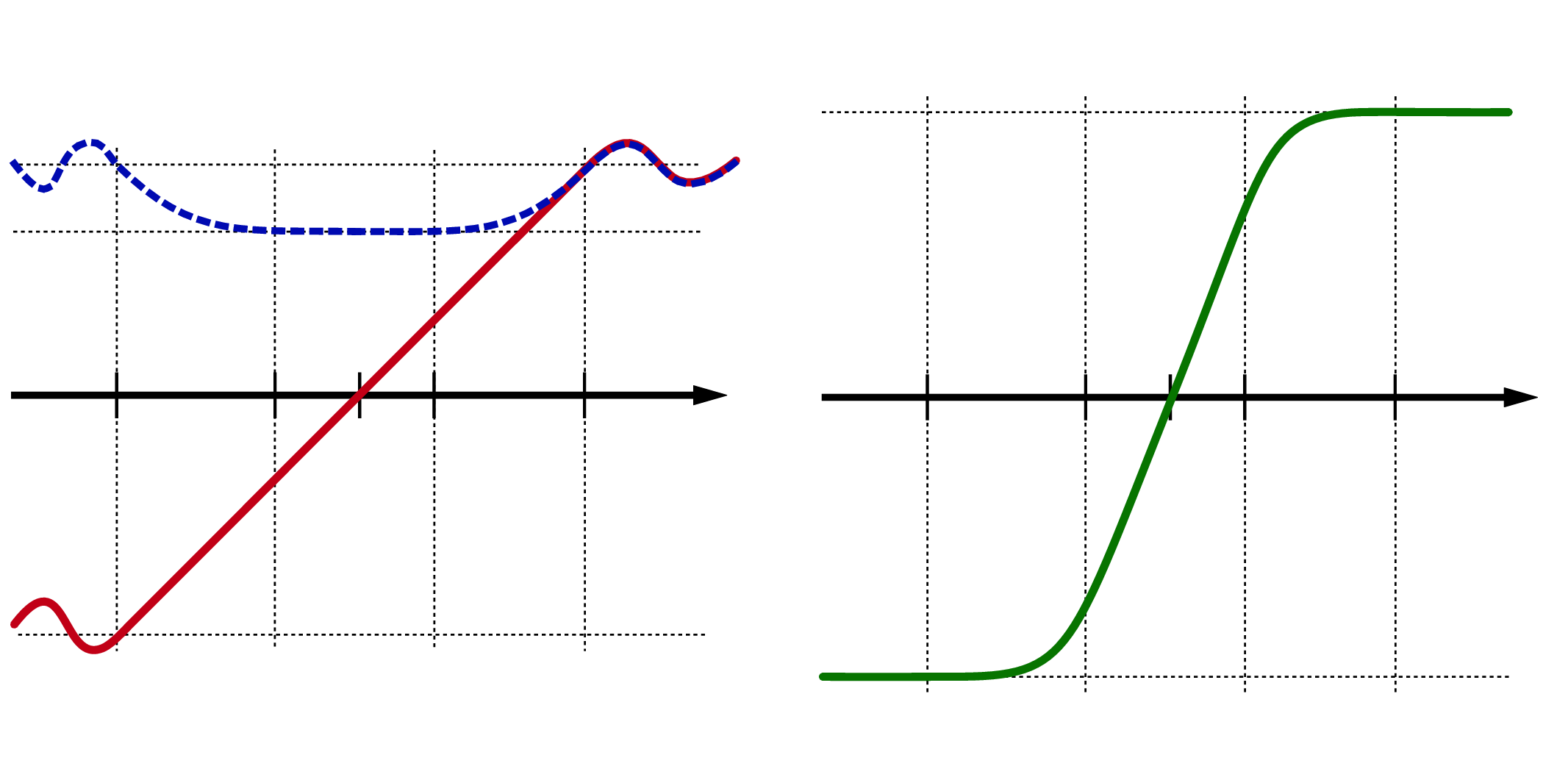}
     \put(46.75,23){\small $\tau$}
     \put(98.5,23){\small $\tau$}
     \put(8,40){\small \textcolor{darkblue}{$\mu(\tau)$}}
     \put(30,28){\small \textcolor{darkred}{$f(\tau)$}}
     \put(64,30){\small \textcolor{darkgreen}{$\tilde{f}(\tau) = \frac{f(\tau)}{\mu(\tau)}$}}

    \put(5.75,5){\small $-\ve$}
    \put(15.5,5){\small $-\delta$}
    \put(27,5){\small $\delta$}
    \put(36.75,5){\small $\ve$}

    \put(57.25,2.5){\small $-\ve$}
    \put(67,2.5){\small $-\delta$}
    \put(78.5,2.5){\small $\delta$}
    \put(88.25,2.5){\small $\ve$}

    \put(-1,33.75){\small $\ve'$}
    \put(-1,37.5){\small $\ve$}
    \put(50.5,41){\small $1$}
     
	\end{overpic}
    \vskip-0.2cm
	\caption{The normalization of $f$ constructed in the proof of \cref{lemma:CHTtoFolded}.}
	\label{fig:function2}
\end{figure}

\vspace{2mm}
\noindent \emph{Step 2: Define $g$.}
\vspace{2mm}

With our contact form normalized as in Step 1, we define $g := e^{-f^2}$ and $\omega := d(g\, \beta)$. We need to prove that $\omega$ is an exact folded symplectic form with fold $\Gamma$ and positive contact-type Liouville form. We have $\omega = dg\wedge \beta + g\, d\beta$ and so 
\begin{equation}
    \omega^n = ng^{n-1}\, dg\wedge \beta \wedge (d\beta)^{n-1} + g^n\, (d\beta)^n.
\end{equation}
Away from $N(\Gamma)$, $f = \pm 1$ on $R_{\pm}$ and so $dg = 0$. Thus, in this region $\omega^n =e^{-n} (d\beta)^n$. By property (1), $\pm \omega^n > 0$ on $R_{\pm} \setminus N(\Gamma)$. On $N(\Gamma)$, we have $dg = -2f(\tau)e^{-f(\tau)^2}\, d\tau$ and so 
\begin{align*}
\omega^n &= -2nf(\tau)e^{-nf(\tau)^2}\, d\tau \wedge \beta_{\Gamma} \wedge (d\beta_{\Gamma})^{n-1} \, + \, e^{-nf(\tau)^2}\, (d\beta_{\Gamma})^n \\
&= 2n f(\tau)e^{-nf(\tau)^2}\, \beta_{\Gamma} \wedge (d\beta_{\Gamma})^{n-1} \wedge d\tau
\end{align*}
where the second equality follows from the fact that $(d\beta_{\Gamma})^n= 0$ on $\Gamma^{2n-1}$ by a dimension count. Since $\beta_{\Gamma}$ is a contact form, $\beta_{\Gamma} \wedge (d\beta_{\Gamma})^{n-1}\wedge d\tau $ is an orienting volume form on $\Gamma \times (-\ve, \ve)_{\tau}$. By our normalization of $f$, it follows that $\omega^n \pitchfork 0$ exactly along $\Gamma$, as desired. 

The fact that $g\, \beta$ is a positive contact-type Liouville form is immediate, since $\iota(g\, \beta) = \beta_{\Gamma}$, the latter of which is a positive contact form by assumption. This completes the proof. 
\end{proof}

\subsection{Proof of \cref{theorem:main}}\label{subsec:main}

Recall that an \textit{almost contact structure} on an odd-dimensional manifold $M^{2n+1}$ is a hyperplane distribution $H\subset TM$ together with a complex bundle structure $J:H \to H$. Equivalently, we may view an almost contact structure as a pair $(\alpha, \omega)$ where $\alpha$ is a non-vanishing $1$-form on $M$ and $\omega$ is a non-degenerate $2$-form on $H:=\ker \alpha$. The following fact is standard (see, for example, \cite{borman2014overtiwsted}), but we include a proof for completeness.

\begin{lemma}\label{lemma:SACStoContact}
Let $\Sigma$ be a closed and oriented manifold of dimension $2n\geq 2$ with a stable almost complex structure. Then $\Sigma \times \R$ has an almost contact structure.     
\end{lemma}

\begin{proof}
Let $J_0$ be a complex (bundle) structure on $T\Sigma \oplus \underline{\ve}^2 \to \Sigma$, where $\underline{\ve}^2 \to \Sigma$ is the trivial rank-$2$ bundle. Let $\pi: \Sigma \times \R \to \Sigma$ be the projection. Then $J:=\pi^*J_0$ is a complex structure on the pullback bundle $\pi^*(T\Sigma \oplus \underline{\ve}^2) \to \Sigma \times \R$. Note that 
\[
\pi^*(T\Sigma \oplus \underline{\ve}^2) \cong T(\Sigma \times \R) \oplus \underline{\ve}^1
\]
where, on the right, $\underline{\ve}^1 \to \Sigma \times \R$ is the trivial rank-$1$ bundle. Let $H$ be the hyperplane distribution of complex tangencies in $T(\Sigma \times \R)$ induced by $J$, i.e., 
\[
H:= J[T(\Sigma \times \R) \oplus 0] \cap [T(\Sigma \times \R) \oplus 0]
\]
viewed as a subbundle of $T(\Sigma \times \R)$. Then $(H, J\mid_{H})$ is an almost contact structure on $\Sigma \times \R$. 
\end{proof}

\begin{proof}[Proof of \cref{theorem:main}]
Let $\Sigma$ be a closed and oriented manifold of dimension $2n\geq 2$ with a stable almost complex structure. By \cref{lemma:SACStoContact}, there is an almost contact structure on $\Sigma \times \R$. By Gromov's $h$-principle for the existence of contact structures on open manifolds \cite{gromov1969stable}, the almost contact structure may be homotoped to a genuine contact structure $\xi$ on $\Sigma \times \R$. 

By \cref{theorem:hh19}, there is an embedding $\iota:\Sigma \hookrightarrow \Sigma \times \R$ which is $C^0$-close to the $0$-inclusion $\iota_0(\Sigma)=\Sigma \times \{0\}$ such that $\iota(\Sigma)$ is a Weinstein convex hypersurface in $(\Sigma \times \R, \xi)$. This means that there is a smooth Morse function $\phi: \iota(\Sigma) \to \R$ with dividing set $\phi^{-1}(0)$ for which the characteristic foliation of $\iota(\Sigma)$ is gradient-like. 

By \cref{lemma:CHTtoFolded}, there is an exact folded symplectic structure $d\lambda$ on $\iota(\Sigma)$ with positive contact-type Liouville form whose fold coincides with the dividing set of $\iota(\Sigma)$. Moreover, the Liouville vector field $X_{\lambda}$ of the folded form directs the characteristic foliation of $\iota(\Sigma)$ on $R_+$, and directs the oppositely oriented characteristic foliation on $R_-$. This means that $X_{\lambda}$ is gradient-like for $\pm \phi\mid_{R_{\pm}}$ and hence $(\lambda, \phi)$ is a folded Weinstein structure on $\iota(\Sigma)$. This pulls back to a folded Weinstein structure on $\Sigma$, as desired. 
\end{proof}

We close this subsection with a discussion on how our proof of \cref{theorem:main} compares and contrasts with the technique of \cite{cannas2010foldedfour} in proving existence of folded symplectic forms. The situation is summarized by \cref{fig:proofcomp}. 

\begin{figure}[ht]
	\begin{overpic}[scale=0.4]{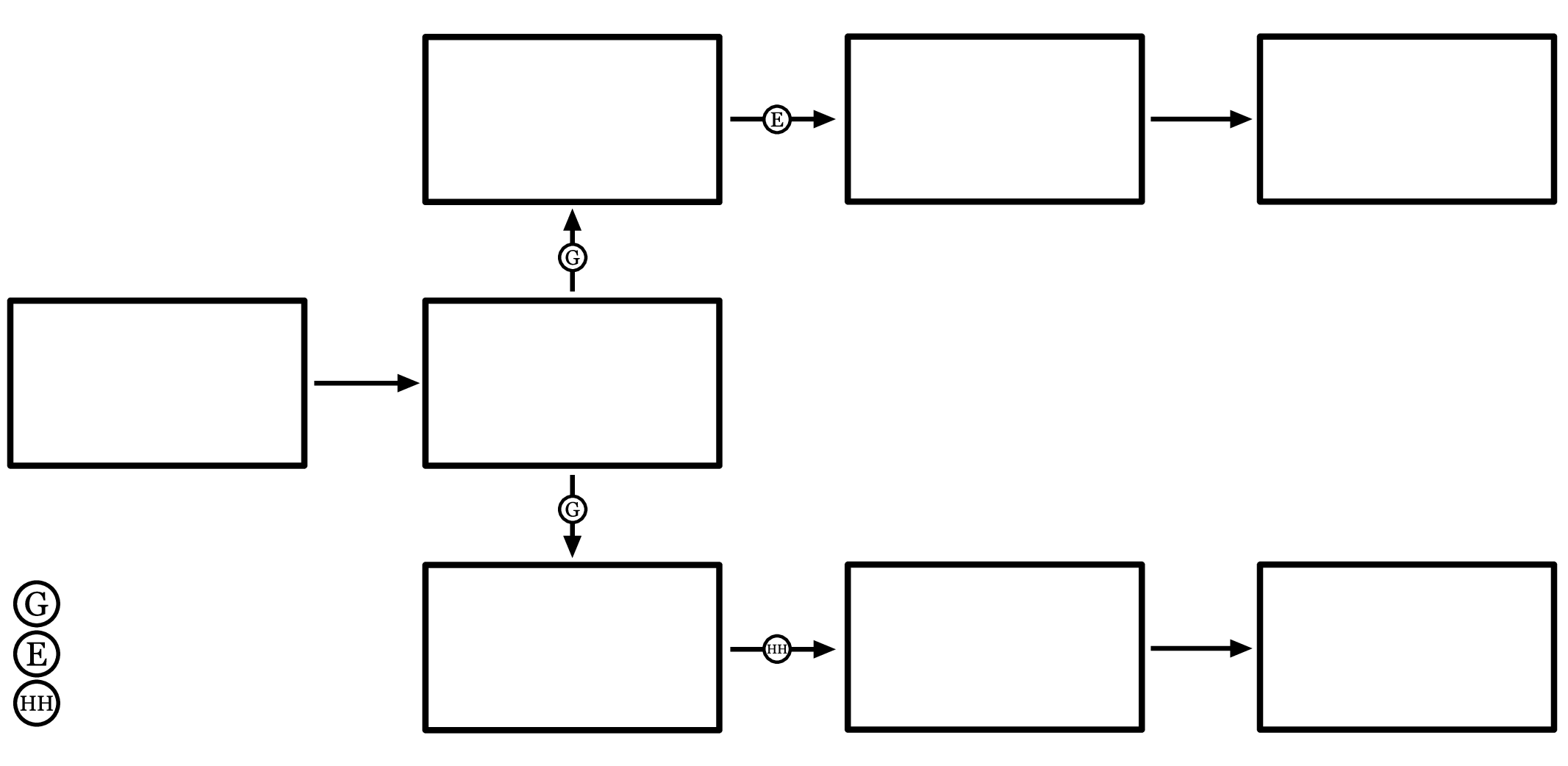}
     \put(4.5, 10){\tiny $=$ Gromov \cite{gromov1969stable}}
     \put(4.5, 6.75){\tiny $=$ Eliashberg \cite{eliashberg1970fold}}
     \put(4.5, 3.5){\tiny $=$ Honda-Huang \cite{honda2019convex}}

     \put(2,25){\footnotesize $\Sigma$ stable almost}
     \put(1.25,22.5){\footnotesize complex structure}

     \put(30,25){\footnotesize $\Sigma\times \R$ almost}
     \put(28,22.5){\footnotesize contact structure}

     \put(30,8.5){\footnotesize $\Sigma\times \R$ contact}
     \put(31,6){\footnotesize structure $\xi$}

     \put(56,8.5){\footnotesize $\Sigma \hookrightarrow (\Sigma\times \R, \xi)$}
     \put(55,6){\footnotesize Weinstein convex}

     \put(85.5,8.5){\footnotesize $\Sigma$ folded}
     \put(81,6){\scriptsize Weinstein structure}
     \put(85.5,42){\footnotesize $\Sigma$ folded}
     \put(82,39.5){\footnotesize symplectic form}

     \put(27.85,43.25){\tiny $\Sigma\times \R$ almost contact}
     \put(35.75,42.25){\small \rotatebox{-90}{$\subseteq$}}
     \put(29,38.5){\tiny \tiny $\Sigma\times \R^2$ symplectic}

     \put(56,44.25){\tiny folded immersion}
     \put(57.25,41.75){\tiny $\Sigma \to (\Sigma \times \R)/L$}
     \put(56.5,39.5){\tiny into symplectic}
     \put(59,37.25){\tiny leaf space}
     
	\end{overpic}
	\caption{The top branch describes the folded symplectic existence argument of \cite{cannas2010foldedfour}, and the lower branch describes the proof of \cref{theorem:main}.}
	\label{fig:proofcomp}
\end{figure}
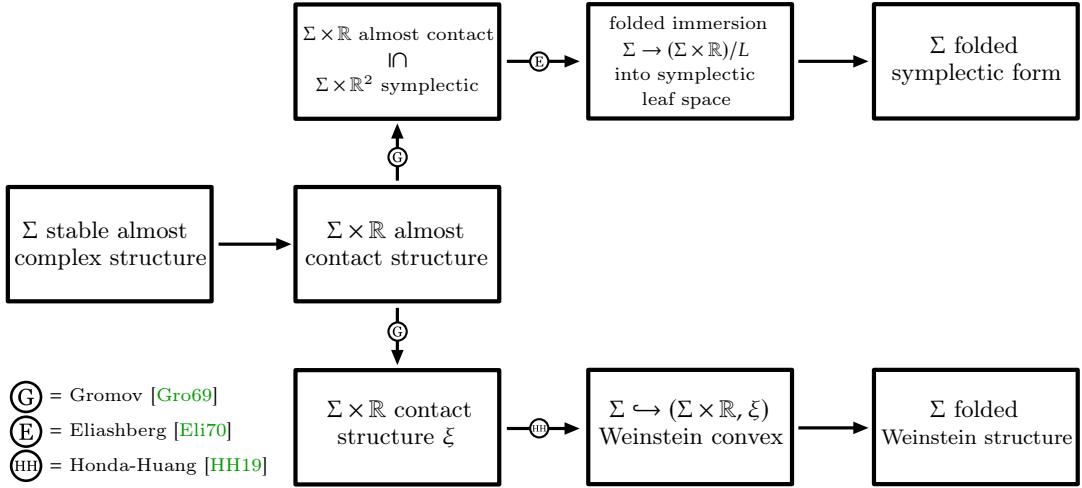

Briefly and informally, Cannas da Silva's existence proof proceeds as follows: 
\begin{enumerate}
    \item Assume that $\Sigma$ is equipped with a stable almost complex structure. From this one can induce an almost complex structure $J$ on $\Sigma \times \R^2$. 
    \item By Gromov's $h$-principle for existence of symplectic structures on open manifolds \cite{gromov1969stable}, after homotoping $J$ there is an honest symplectic structure $\omega$ compatible with $J$. Moreover, as in \cref{lemma:SACStoContact}, $J$ induces a co-orientable almost contact structure on the codimension-$1$ submanifold $\Sigma \times \R$. 

    \item Let $\iota:\Sigma \times \R \hookrightarrow \Sigma \times \R^2$ be the inclusion. Since $\Sigma \times \R^2$ is symplectic, the leaf space $(\Sigma \times \R)/L$ of the rank-$1$ foliation $L$ determined by $\ker \iota^*\omega$ is locally symplectic. One then applies Eliashberg's $h$-principle for folded immersions relative to a foliation \cite{eliashberg1970fold} to $\Sigma \to \Sigma \times \R$ relative to $L$ to pull back the locally symplectic structure on the leaf space to a folded symplectic form on $\Sigma$.  
\end{enumerate}

\subsection{Proof of \cref{thm:lefschetz}.}\label{subsec:lefschetzproof}

Finally we establish existence of supporting folded Weinstein Lefschetz fibrations on all folded Weinstein manifolds. 

\begin{proof}[Proof of \cref{thm:lefschetz}.]
Let $(\Sigma, \lambda, \phi)$ be a folded Weinstein manifold with fold $(\Gamma, \xi_{\Gamma})$ and positive (resp. negative) region $R_{\pm}$. Let $(\pm R_{\pm}^{\ve}, \lambda, \pm\phi)$ be the Weinstein domain given by a slight compact retract of $R_{\pm}$ as in \cref{remark:retract_domain}.     

By \cref{thm:GP17} of Giroux and Pardon \cite{giroux2017lefschetz}, there is a Weinstein Lefschetz fibration $(W_{\pm}; \tilde{\mathcal{L}}^{\pm})$ whose total space $|W_{\pm}; \tilde{\mathcal{L}}^{\pm}|$ is Weinstein homotopic to $(\pm R_{\pm}^{\ve}, \lambda, \pm\phi)$. Note that the Weinstein domains $R_{+}^{\ve}$ and $-R_{-}^{\ve}$ each have boundaries contactomorphic to $(\Gamma, \xi_{\Gamma})$. Thus, by the Giroux correspondence \cite{BHH23}, the abstract open book decompositions $(W_+, \tau_{\tilde{\mathcal{L}}^{+}})$ and $(W_-, \tau_{\tilde{\mathcal{L}}^{-}})$ may be stabilized to a common abstract open book $(W_0, \psi)$. Each of the open book stabilizations is induced by a Lefschetz fibration stabilization, and hence we have 
\[
\psi = \tau_{\mathcal{L}_0^+} \circ \tau_{\tilde{\mathcal{L}}^{+}} = \tau_{\mathcal{L}_0^-} \circ \tau_{\tilde{\mathcal{L}}^{-}}
\]
where $\mathcal{L}_0^{\pm}$ is the tuple of additional vanishing cycles in $W_0$ involved in stabilizing $(W_{\pm}, \tau_{\tilde{\mathcal{L}}^{\pm}})$ to $(W_0, \psi)$. This means that the total space
\[
|W_0; \mathcal{L}^{\pm} := \mathcal{L}_0^{\pm} \cup \tilde{\mathcal{L}}^{\pm}|
\]
is also Weinstein homotopic to $(\pm R_{\pm}^{\ve}, \lambda, \pm\phi)$. 

It follows that the folded Weinstein Lefschetz fibration $(W_0; \mathcal{L}^+, \mathcal{L}^-)$ has total space folded Weinstein homotopic to the double of $(R_{+}^{\ve}, \lambda, \phi)$ and $(-R_{-}^{\ve}, \lambda, -\phi)$, which is folded Weinstein homotopic to $(\Sigma, \lambda, \phi)$. This completes the proof. 
\end{proof}

\bibliography{references}
\bibliographystyle{amsalpha}

\end{document}